%% file: main.tex
\documentclass[11pt, a4paper]{amsart}
\usepackage[a4paper]{geometry}
\usepackage[utf8]{inputenc}
\usepackage{amsmath}
\usepackage{amssymb}
\usepackage{amsthm}
\usepackage{enumerate}
\usepackage{graphicx}
\usepackage{import}
\usepackage{xifthen}
\usepackage{pdfpages}
\usepackage{tikz} 
\usepackage{tikz-cd} 
\usepackage{mathabx}
\usepackage{hyperref}
\usepackage{thmtools, thm-restate}

\usepackage[foot]{amsaddr}

\usetikzlibrary{shapes.geometric}
\usetikzlibrary{graphs,graphs.standard}
\usetikzlibrary{fit}
\usetikzlibrary{calc}

\usepackage[giveninits]{biblatex}
\theoremstyle{plain}
\newtheorem{theorem}{Theorem}[section]

\newtheorem{lemma}[theorem]{Lemma}
\newtheorem{corollary}[theorem]{Corollary}
\theoremstyle{definition}
\newtheorem{definition}[theorem]{Definition}
\newtheorem{example}[theorem]{Example}

\newtheorem{question}[theorem]{Question}

\theoremstyle{remark}
\newtheorem*{remark}{Remark}

\newcommand{\qi}{quasi-isometric }

\newcommand{\Z}{\mathbb{Z}}

\newcommand{\Img}{\operatorname{Im}}
\newcommand{\Pqs}{\operatorname{PQStab}}
\newcommand{\Ps}{\operatorname{PStab}}
\newcommand{\lk}{\operatorname{link}}
\newcommand{\diam}{\operatorname{diam}}
\newcommand{\supp}{\operatorname{supp}}
\newcommand{\esupp}{\operatorname{esupp}}
\newcommand{\lv}{\operatorname{l_{V_\Gamma}}}
\newcommand{\fv}{\operatorname{f_{V_\Gamma}}}
\newcommand{\Pc}{\operatorname{Pc_\Gamma}}

\title{Groups Acting Acylindrically on Trees}
\author{William D Cohen}
\address{Centre for Mathematical Sciences, University of Cambridge, 
Cambridge, CB3 0WA}
\email{wdc26@cam.ac.uk}

\AtBeginBibliography{\small}

\begin{document}
\begin{abstract} We develop a notion of groups that act acylindrically and non-elementarily on simplicial trees, which we call acylindrically arboreal groups. We then prove a complete classification of when graph products of groups and the fundamental groups of certain hyperbolic $3$-manifolds are acylindrically arboreal, and use these classifications to provide examples of acylindrically hyperbolic groups that have actions on trees but have no non-elementary acylindrical actions on trees.
\end{abstract}

\maketitle
\input{introduction.tex}
\input{preliminaries.tex}
\input{graphProducts.tex}

\input{manifolds.tex}
\printbibliography
\end{document}

%% file: introduction.tex
\section{Introduction}
The definition of an acylindrical action on a tree was first formulated by Sela \cite{Sela97}, and later generalised by Weidmann \cite{Weidmann07} to the following. 

\begin{definition}\label{Def:introkc} Let $G$ be a group acting by simplicial isometry on some simplicial tree $T$ and let $k\geq0$ and $C>0$ be integers. We say that the action of $G$ on $T$ is $(k, C)$\emph{-acylindrical} if the pointwise stabiliser of any edge path in $T$ of length at least $k$ contains at most $C$ elements.
\end{definition}

A third definition was formulated by Bowditch \cite{Bowditch08} to give a property of an action on a general metric space, which is more coarse-geometric. The definition of Bowditch can be shown to agree with that of Weidmann in the context of actions on trees, in a result essentially due to Minasyan and Osin \cite{minasyanOsin13}(c.f. Theorem~\ref{Thm:KCisAA}). Groups that act acylindrically non-elementarily (c.f. Theorem~\ref{Thm:categorisedActions}) on hyperbolic spaces are a well-studied generalisation of hyperbolic groups. Such groups are called \emph{acylindrically hyperbolic}.

The class of acylindrically hyperbolic groups is very broad -- such a large number and variety of groups are acylindrically hyperbolic that it can be difficult to discern whether they have certain properties in general. For example, it is still an open problem whether the property of being acylindrically hyperbolic is preserved under quasi-isometry \cite[Question~2.20(a)]{Osin18} or even passes to finite index overgroups \cite[Question 2]{minasyanOsin19}, although some progress was made on the latter question by Balasubramanya \cite{BalasubramanyaFiniteIndex}.

It is natural to ask which acylindrically hyperbolic groups act acylindrically non-elementarily on simplicial trees. Indeed, it is a well-known result of Balasubramanya that all acylindrically hyperbolic groups admit acylindrical actions on quasi-trees \cite{Bala16}, and further restricting to trees allows us to make use of a wide range of strong theory. All actions will be by isometry, and all trees will be assumed to be simplicial with the natural edge metric. We make the following definition.

\begin{definition}
    Let $G$  be a group. Then we say that $G$ is \emph{acylindrically arboreal} if $G$ admits a non-elementary acylindrical action on some tree $T$.
\end{definition}

Many groups are already known to be acylindrically arboreal, although we believe that this terminology is novel. For example Wilton and Zalesskii show that if a closed and orientable irreducible 3-manifold $M$ is non-geometric then the splitting of $\pi_1(M)$ given by the JSJ-decomposition of $M$ induces an acylindrical action on the associated Bass--Serre tree \cite[Lemma~2.4]{WiltonZalesskii08}, so such groups will either be acylindrically arboreal or virtually cyclic by \cite[Theorem~1.1]{Osin13} (c.f. Theorem~\ref{Thm:categorisedActions}). 

The vast majority of other examples arise from a result of Gitik, Mj, Rips and Sageev \cite[Main Theorem]{GitikMahanRipsSageev98} that implies that a splitting of a hyperbolic group over a quasi-convex subgroup will induce an acylindrical action on the associated Bass--Serre tree. The study of quasi-convex splittings of hyperbolic groups has proved very fruitful (see \cite{HaglundWise12, WiseQCBook}, for example).

The acylindrical action on a tree rather than on a generic hyperbolic metric space makes studying the properties of acylindrically arboreal groups much simpler than studying the properties of acylindrically hyperbolic groups in general. For example, the following are properties of acylindrically arboreal groups that are either unknown for acylindrically hyperbolic groups in general (even for those which split) or have a much more satisfying solution in our restricted class.

\begin{itemize}
    \item \textbf{\textit{Quasi-isometry : }}It is very simple to see that acylindrical arboreality is not preserved by quasi-isometry. For example, the $(2, 3, 7)$-triangle group has property (FA) (c.f. Definition~\ref{def:FA}) by a well-known result of Serre (\cite[Corollary 2 of Theorem I.6.26]{stilwell2002trees}, c.f. Theorem~\ref{thm:eliptGen}), and this property will naturally prohibit any action on a tree from fulfilling the non-elementary requirement in the definition of acylindrical arboreality. However, this group is virtually a hyperbolic surface group, which is acylindrically arboreal by considering cuts along simple closed curves.
    \item \textbf{\textit{Growth Rates : }}Using the fact that in a non-elementary acylindrical action on a tree we are guaranteed to find loxodromic elements in words of length at most two over any generating set (c.f Corollary~\ref{cor:eliptgen}), we can learn much about the growth properties of acylindrically arboreal groups. In particular, Kerr \cite[Proposition~1.0.9]{Kerr2021ProductSG} and Fujiwara \cite[Theorem~1.1]{Fujiwara21} prove strong results about the growth rates of acylindrically hyperbolic groups where we can find loxodromics quickly. It is however worth noting that the Fujiwara result requires equational Noetheriality in an essential way, and it is still very much open when an acylindrically arboreal group is equationally Noetherian (see \cite[Theorem~1.9]{Valiunas20} for a strong condition under which equational Noetheriality will hold).
    \item \textbf{\textit{Explicit Constructions in Bounded Cohomology : }}Monod and Shalom explicitly use the acylindrical action on a tree to show that all acylindrically arboreal groups are in the class $\mathcal{C}_{\operatorname{reg}}$ (see \cite[Theorem~7.7, Corollary~7.10 and Remark~7.11]{MonodShalom04}), which is equivalent to the second bounded cohomology with coefficients in the regular representation not vanishing (see \cite{MonodShalom06} for a formal treatment of this property). Their proof is via an explicit construction of the desired non-boundary cocycle, and although subsequent results have shown that all acylindrically hyperbolic groups are in $\mathcal{C}_{\operatorname{reg}}$ (see \cite[Corollary~B]{Hamenstadt08}, \cite[Corollary~1.7]{HullOsin13}), the constructions in the general case are much less explicit. 
\end{itemize}

An acylindrical action on a tree also has strong implications for the Farrel--Jones conjecture, as shown by Knopf \cite{Knopf19}.

Our main results are complete classifications of two very important classes of groups, those of graph products of groups and fundamental groups of compact and orientable hyperbolic 3-manifolds with empty or toroidal boundary. The acylindrical hyperbolicity of groups in these classes was considered in \cite{minasyanOsin13}, and while all manifolds considered in this paper have relatively hyperbolic fundamental groups which must then be acylindrically hyperbolic \cite[Theorem~1.2 and Proposition~2.12]{Osin13}, the results of this paper relating to graph products of groups should be compared to those in \cite[Section~2.3]{minasyanOsin13}.

First, if $\mathbb{GP}(\Gamma, \mathcal{G})$ is a graph product we will say that a pair of vertices $a$ and $b$ of $\Gamma$ are \emph{separated} (with respect to the graph product $\mathbb{GP}(\Gamma, \mathcal{G})$) if the $\Gamma$-edge distance between $a$ and $b$ is at least $2$ and the subgroup generated by the vertex groups corresponding to $\lk_\Gamma(\{a, b\})$ is finite, or equivalently that the set of vertices adjacent to both $a$ and $b$ induces a complete subgraph of $\Gamma$ and are all labelled with finite groups.

\begin{restatable}{theorem}{introGPthm}
\label{Thm:mainIntro} Let $G=\mathbb{GP}(\Gamma, \mathcal{G})$ be a non-degenerate graph product of groups such that $\diam(\Gamma)\geq 2$. Then $G$ is acylindrically arboreal if and only if $G$ is not virtually cyclic and there exists a pair of vertices $a, b\in V(\Gamma)$ that are separated.
\end{restatable}

Here $\diam(\Gamma)$ is the graph theoretical diameter (c.f. Definition~\ref{def:diam}). It is worth noting here that we have no restrictions on the vertex groups of $G=\mathbb{GP}(\Gamma, \mathcal{G})$, and the condition that $G$ is not virtually cyclic is not overly restrictive. Indeed, if $\diam(\Gamma)\geq 2$ then we have that $G$ is virtually cyclic if and only if $\Gamma$ is a complete graph minus one edge $e$, all vertex groups of $\Gamma$ are finite and the endpoints of $e$ are both copies of $\Z/2\Z$, and more generally a group that is virtually cyclic cannot admit a non-elementary action on any hyperbolic space, even allowing for non-acylindrical actions, as it cannot contain multiple independent loxodromic elements. 

In a similar vein, we prove a similar sufficient condition in Corollary~\ref{cor:subgrpsGP} for a subgroup of a graph product of groups to be acylindrical arboreal, which should be compared with \cite[Theorem~2.12]{minasyanOsin13}.

A group may fail to be acylindrically arboreal by having no action on a tree that is sufficiently complex. We make the following definition.
\begin{definition}
    Let $G$ be a group. We say that $G$ has the \emph{weakened (FA) property}, denoted \emph{(FA$^-$)}, if whenever $G$ acts on a tree $T$ it must either fix some point $p\in T$ (not in the boundary of $T$) or fix some bi-infinite geodesic $L\subset T$ setwise.
\end{definition}
This is similar to the definition of property A$\mathbb{R}$ defined by Culler and Vogtmann \cite{Culler1996AGT} for groups acting on real trees, excepting that they only allow the action of the group on the fixed line to be by translation.

We will see that property (FA$^-$) prohibits acylindrical arboreality in the sense that a group with (FA$^-$) can never fulfil the non-elementary condition when acting upon a tree, but that the reverse implication is not true --- there exist many acylindrically hyperbolic (or even hyperbolic) groups that do not have (FA$^-$) but which are not acylindrically arboreal. For example, let $Q$ be a finitely presented but not hyperbolic group with an action on a tree with no fixed points or lines, such as the Baumslag--Solitar group $\operatorname{BS}(2, 3)\cong \langle a, b\mid ab^2a^{-1}b^{-3}\rangle$. Using the Rips construction of Ollivier and Wise \cite[Theorem~1.1]{OlivierWise04} it is possible to extend $Q$ by an infinite group $N$ with Kazhdan's property (T) (see \cite{Zuk03}, for example) to obtain a hyperbolic group $G$, which will have an infinite normal (T) (and therefore (FA$^-$) by \cite{Watatani}, or \cite[Theorem B]{NibloReeves97}) subgroup. Such a subgroup prohibits the acylindrical arboreality of $G$ by the fact that as with acylindrical hyperbolicity, acylindrical arboreality is inherited by infinite normal subgroups (Lemma~\ref{lem:inheritance}(3), see \cite[Lemma~7.1]{Osin13} for the original proof for acylindrically hyperbolic groups). However, $G/N\cong Q$ admits an action on a tree with no fixed points or lines by construction, so $G$ will not have (FA)$^-$ as it will act on the same tree via the quotient map.

Using Theorem~\ref{Thm:mainIntro} we will obtain the following result.
\begin{restatable}{proposition}{brokenraag}
\label{prop:brokenRAAG}
   There exists an acylindrically hyperbolic right angled Artin group $G$ that is not acylindrically arboreal but does not have property (FA$^-$). Furthermore, we can construct $G$  such that it has no non-trivial normal subgroups with property (FA$^-$).
\end{restatable}

Our second main theorem, Theorem~\ref{thm:mainHypMfd} below, provides a complete classification of when the fundamental group of a compact and orientable hyperbolic 3-manifold with empty or toroidal boundary is acylindrically arboreal --- see Section~\ref{sec:hyperbolic} for the relevant definitions.

\begin{restatable}{theorem}{intromfdclass}
\label{thm:mainHypMfd}
	Let $M$ be a compact and orientable hyperbolic 3-manifold with empty or toroidal boundary. Then $\pi_1(M)$ is acylindrically arboreal if and only if $M$ contains an embedded $2$-sided incompressible closed subsurface $\Sigma$ that is not isotopic to any boundary component of $M$, and such that the image of the natural inclusion $\pi_1(\Sigma)\hookrightarrow\pi_1(M)$ is geometrically finite.
\end{restatable}

Along with the proof of a result of Minasyan and Osin \cite[Theorem~3]{minasyanOsin19}, Theorem~\ref{thm:mainHypMfd} restricted to the closed case may be used to give a complete classification of the acylindrical arboreality of the fundamental groups of closed and orientable $3$-manifolds.

Many hyperbolic groups are already known to be acylindrically arboreal, but it is unknown in general whether all acylindrically arboreal hyperbolic groups admit a quasi-convex or malnormal splitting. We thus ask the following question.

\begin{question}\label{Question2}
    Let $G$ be an acylindrically arboreal hyperbolic group. Does $G$ necessarily admit a non-elementary splitting over either a quasi-convex or finitely generated malnormal subgroup?
\end{question}

Using Theorem~\ref{thm:mainHypMfd} we provide a partial answer to this question.

\begin{restatable}{theorem}{intromfd}
\label{thm:intromfd}
    Let $G=\pi_1(M)$ be the fundamental group of a closed and orientable hyperbolic 3-manifold. Then the following are equivalent:
    \begin{enumerate}
        \item The group $G$ admits a non-elementary quasi-convex splitting;
        \item The group $G$ is acylindrically arboreal;
        \item The group $G$ does not have property (FA$^-$).

    \end{enumerate}
\end{restatable}
An example of an acylindrically arboreal hyperbolic group with no quasi-convex splitting would require a rather pathological construction, although we are inclined to believe that such an example will exist. For example, the double of a hyperbolic group $G$ over a finitely presented subgroup $H$ will be hyperbolic only if $H$ is undistorted \cite[Theorem I.$\Gamma$.6.20]{BridsonHaefliger}, so a counterexample using the doubling construction would require a subgroup of a hyperbolic group which is finitely generated but not finitely presented. Such examples are poorly understood in general, and tend to have large normalisers in the ambient group \cite{WiseCoherence}.
\subsection*{Acknowledgements}
This work was completed while the author was a PhD student at the University of Cambridge, supervised by Dr Jack Button. The author is very grateful to his supervisor for suggesting the direction of this research and for all of his help, and to Macarena Arenas and Francesco Fournier-Facio for their helpful comments. The author is also very grateful to two anonymous reviewers for their helpful comments, in particular for the suggestion of a more streamlined statement and proof for Lemma~\ref{lem:AAimpliesFinite}, and for advice on how to make Section~\ref{sec:hyperbolic} clearer and more rigorous. Finally, financial support from the Cambridge Trust Basil Howard Research Graduate Studentship is gratefully acknowledged. 

%% file: preliminaries.tex
\section{Preliminaries}\label{sec:prelim}
\input{preliminariesAH.tex}
\input{preliminariesAAKC}

%% file: preliminariesAH.tex
\subsection{Acylindrical Hyperbolicity}
We will define a hyperbolic metric space to be a geodesic metric space satisfying any of the standard equivalent conditions, for example the thin triangles condition for $\delta$-hyperbolicity due to Rips \cite[Definition III.H.1.1]{BridsonHaefliger}. We have the following coarse-geometric definition of acylindricity due to Bowditch.
\begin{definition}\label{def:acyl}
   Let $G$ be a group acting on a metric space $(X, d)$ by isometry. For all $\epsilon\geq 0$, $x, y\in X$ we define the \emph{Pointwise Quasi-Stabiliser} $\Pqs_G^{\epsilon}(\{x, y\})$ to be the set 
    $\{g\in G\mid d(x, gx)\leq\epsilon\text{ and }d(y, gy)\leq\epsilon\}.$
    We say that $G$ acts \emph{acylindrically} on $X$ if for all $\epsilon\geq0$ there exist constants $R(\epsilon)$, $N(\epsilon)\geq0$ such that for all $x, y\in X$ with $d(x, y)\geq R(\epsilon)$, we have $|\Pqs_G^{\epsilon}|\leq N(\epsilon)$. 
\end{definition}
We have the following ways in which elements of a group can act upon a hyperbolic space.
\begin{definition}\label{def:isotype}
    Let $G$ be a group acting on a hyperbolic space $X$ by isometry, and let $g\in G$. Then we say that
    \begin{itemize}
        \item $g$ is \emph{elliptic} if $g$ has bounded orbits;
        \item $g$ is \emph{loxodromic} if for any $x\in X$ the map $\mathbb{Z}\rightarrow X$ given by $n\mapsto g^{n}x$ is a \qi embedding;
    \end{itemize}
\end{definition}
We will require some results pertaining to acylindrically hyperbolic groups, which we will include without proof for brevity. We first have the following classification theorem due to Osin. A similar classification for generic actions on hyperbolic spaces was originally proved by Gromov in \cite[Section~8]{Gromov87}, with more possibilities, but these may be disregarded as they will not occur in the acylindrical case.
\begin{theorem}
    \textit{\cite[Theorem~1.1]{Osin13}}\label{Thm:categorisedActions} Let $G$ be a group acting acylindrically on a hyperbolic space $X$. Then $G$ satisfies exactly one of the following conditions.
    \begin{enumerate}
        \item The orbit of any element $x\in X$ under the action of $G$ is bounded. In this case we say that the action of $G$ on $X$ is \emph{elliptic}.
        \item The group $G$ is virtually cyclic and contains at least one loxodromic element. In this case we say that the action of $G$ on $X$ is \emph{lineal}.
        \item The group $G$ contains infinitely many independent loxodromic elements. In this case we say that action of $G$ on $X$ is \emph{non-elementary}.
    \end{enumerate}       
\end{theorem}
\begin{definition}
    Let $G$ be a group. We say that $G$ is \emph{acylindrically hyperbolic} if $G$ admits a non-elementary acylindrical action on a hyperbolic space $X$.
\end{definition}

\begin{example} It is well known that any finitely generated hyperbolic group acts acylindrically on its Cayley graph, which is a hyperbolic metric space. This action will have bounded orbits if and only if $G$ is finite, and so if $G$ is an infinite and non-virtually cyclic hyperbolic group it must be acylindrically hyperbolic by Theorem~\ref{Thm:categorisedActions}. 

Further examples include the mapping class groups MCG$(\Sigma_{g, p})$ of closed surfaces of genus $g$ with $p$ punctures unless $g=0$ and $p\leq 3$, $\operatorname{Out}(F_n)$ for $n\geq 2$ and one-relator groups with at least three generators (see \cite[Section~8]{Osin13}, for example).
\end{example}

%% file: preliminariesAAKC.tex
\subsection{Groups Acting on Trees}
We recall some graph theoretical notation, which we will use throughout this paper.
\begin{definition}\label{def:graph}
    Let $\Gamma=(V(\Gamma), E(\Gamma))$ be a graph. Then $\Gamma$ is said to be\emph{finite} if $|V(\Gamma)|<\infty$, and we say that $\Gamma$ is \emph{simple} if $E(\Gamma)$ contains no loops or multiedges.

    For a vertex $v$ of a finite simple graph $\Gamma$ we define the \emph{link}, denoted $\lk_\Gamma(v)$, of $v$ to be the set of vertices $u\in V(\Gamma)\backslash\{v\}$ such that there exists an edge $e\in E(\Gamma)$ incident on both $u$ and $v$. For a subset $A$ of $V(\Gamma)$ we define $\lk_\Gamma(A)$ to be the intersection $\lk_\Gamma(A)=\bigcap_{v\in A}\lk_\Gamma(v)$. We define the \emph{neighbourhood} of a vertex $v\in V(\Gamma)$ to be $N_\Gamma(v)=\lk_\Gamma(v)\cup v$, and the neighbourhood of a set of vertices $A$ to be the union 
    $N_\Gamma(A)=\bigcup_{v\in A}N(v)$.
\end{definition} 

\begin{example} We will refer to the following standard collections of graphs.
\begin{enumerate}
        \item We say that $\Gamma=(V, E)$ is a \emph{complete} graph if $E$ contains every possible unordered pair of distinct elements in $V$, and we say that $\Gamma$ is \emph{discrete} if the edge set $E$ is empty. If $|V|=n$ we denote these graphs as $K_n$ and $O_n$ respectively.
        \item We define the \textit{$n$-path} $P_n$ for $n\geq 2$ to be the unique (up to isomorphism) connected graph on $n$ vertices with $n-1$ edges and maximum vertex degree two, and the \textit{$n$-cycle} $C_n$ for $n\geq 3$ to be the unique (up to isomorphism) connected graph on $n$ vertices with $n$ edges such that the degree of every vertex is two. 
    \end{enumerate}
\end{example}
\begin{definition}\label{def:diam}
    Let $\Gamma=(V, E)$ be a finite graph. For a graph $\Gamma=(V, E)$ we define the \emph{graph theoretical diameter} $\diam(\Gamma)$ to be the metric diameter of the $V(\Gamma)$ when endowed with the edge metric. We will therefore have that $\diam(\Gamma)=\infty$ if and only if $\Gamma$ is disconnected.

    For a subset $A\subseteq\Gamma$ we define the diameter $\diam(A)$ to be the diameter of the subgraph of $\Gamma$ induced by $A$, which can be infinite even if $\Gamma$ had finite diameter.
\end{definition}
\begin{remark}
    The graph theoretical diameter will often differ from the metric diameter of the entire graph. For example, the metric diameter of the cycle $C_5$ is $2.5$, but $\diam(C_5)=2$.
\end{remark}

We will assume the reader has some familiarity with Bass--Serre theory, and for a more detailed discussion we refer to \cite{stilwell2002trees, dicks2011groups}. Let $(\Gamma, \mathfrak{G})$ be a \emph{graph of groups}, where $\Gamma$ is a connected directed graph that may not be finite or simple and $\mathfrak{G}$ is the following data:
\begin{itemize}
    \item To every vertex $v\in V(\Gamma)$ we assign a \emph{vertex group} $G_v$, and to every edge $e\in E(\Gamma)$ we assign an \emph{edge group} $G_e$;
    \item To every edge $e\in E(\Gamma)$ we assign monomorphisms \[d_0:G_e\rightarrow G_{i(e)}\text{ and }d_1:G_e\rightarrow G_{t(e)},\] where $i(e)$ and $t(e)$ are the initial and terminal vertices of $e$ in $\Gamma$ respectively.
\end{itemize}

We will use a slight abuse of notation to consider each vertex group $G_v$ as a subgroup of the fundamental group $\pi_1(\Gamma, \mathfrak{G})$ along the natural inclusion. Similarly, we will consider each edge group $G_e$ to be the subgroup of the fundamental group given by the image of $d_0(G_e)$ in the vertex group $G_{i(e)}$. We call a graph of groups \emph{trivial} if there exists some $v\in V(\Gamma)$ such that $G_v= \pi_1(\Gamma, \mathfrak{G})$, or \emph{non-trivial} otherwise. We say that a graph of groups $(\Gamma, \mathfrak{G})$ is a \emph{graph of groups decomposition}  or \emph{splitting} of a group $G$ if the fundamental group $\pi_1(\Gamma, \mathfrak{G})$ is isomorphic to $G$. We denote by $T(\Gamma, \mathfrak{G})$ the \emph{Bass--Serre tree} associated to the splitting, on which $G$ acts naturally by isometry with respect to the edge metric and without inversion \cite[Section~I.5.3]{stilwell2002trees}.

The assumption that any action on a tree is simplicial and without inversion is easy to guarantee, so we will assume from now on that all actions on trees are by simplicial isometry and without inversion.

As in \cite[Section~I.5.4]{stilwell2002trees}, an action on a tree will give rise to a \emph{quotient graph of groups decomposition $(T/G, \mathfrak{G}$)} of $G$, where the vertex or edge group of a vertex or edge of $T/G$ is defined to have the isomorphism type of the stabiliser of any preimage of that vertex or edge in $T$, and the edge monomorphisms are defined similarly.

The following lemma is a restating of \cite[Proposition~I.2.10]{stilwell2002trees} and its corollary which allows us to remove any ambiguity as to the definition of an elliptic action in the context of simplicial trees and which we include without proof.
\begin{lemma}\label{lem:elipticWellDef}
    Let $G$ be a group acting on a tree $T$. Then the following conditions are equivalent:
    \begin{enumerate}
        \item The $G$-orbit of at least one point $x\in T$ is bounded in $T$; \item The $G$-orbit of every point in $x\in T$ is bounded in $T$; and
        \item There exists some point $x\in T$ fixed by the action of $G$ on $T$.
    \end{enumerate}
\end{lemma}

We use this to make the following definition.
\begin{definition}
    Let $G$ be a group acting on a tree $T$. We say that $G$ is acting \emph{elliptically}, or that $G$ is \emph{elliptic}, if the action of $G$ on $T$ has a fixed point.

    We say that a subgroup $H\leq G$ acts \emph{elliptically} on $T$, or that $H$ is an \emph{elliptic} subgroup, if the induced action of $H$ on $T$ is elliptic. Similarly, we say that an element $g\in G$ acts \emph{elliptically} on $T$, or that $g$ is an \emph{elliptic} element, if $\langle g\rangle\leq G$ is an elliptic subgroup. 
\end{definition}

It follows from the fundamental theorem of Bass--Serre theory that an action on a tree is elliptic if and only if the quotient graph of groups is trivial.
\begin{remark}
    In the context of acylindrical actions on trees these definitions will agree with those in Definition~\ref{def:isotype} and Theorem~\ref{Thm:categorisedActions} by the above lemma, although it is important to note that in general elliptic acylindrical actions on hyperbolic spaces (in the sense of Theorem~\ref{Thm:categorisedActions}) need not have fixed points.
\end{remark}

We will refer to the following well-known result from Bass--Serre Theory without proof.
\begin{theorem}
    \label{thm:eliptGen} \textit{\cite[Corollary 2 of Theorem I.6.26]{stilwell2002trees}} Let $G$ be a group generated by a finite number of elements $s_1,...,s_m$ and let $T$ be a tree on which $G$ acts by isometry. Assume further that for all distinct pairs $1\leq i, j\leq m$ we have that $s_i, s_j$ and $s_is_j$ act elliptically on $T$. Then $G$ must act elliptically on $T$.
\end{theorem}
The following corollary is immediate.
\begin{corollary} \label{cor:eliptgen}
	Let $G$ be a group acting on a tree $T$ generated by an arbitrary set of elements $S$. Assume further that for all distinct pairs $s_i, s_j\in S$ we have that $s_i, s_j$ and $s_is_j$ act elliptically on $T$. Then the action of $G$ on $T$ has no loxodromic elements.
\end{corollary}
Note that this will hold even in the case when $S$ is uncountable, as any loxodromic element would have to appear as a finite word over the elements of $S$, so would sit in a finitely generated subgroup where we can apply Theorem~\ref{thm:eliptGen}.

By Theorem~\ref{Thm:categorisedActions}, Corollary~\ref{cor:eliptgen} implies that if the action is also acylindrical then $G$ must act elliptically, so must have a global fixed point.

We have the following definition due to Serre \cite[Section~I.6.1]{stilwell2002trees}.
\begin{definition}\label{def:FA}
    We say that a group $G$ has \emph{property (FA)} if every action of $G$ on a tree is elliptic, or equivalently, if every graph of groups decomposition of $G$ is trivial.
\end{definition}
In the context of acylindrical arboreality lineal actions, those actions that fix bi-infinite geodesics setwise, are also prohibited. This leads us to the following definition.
\begin{definition}
    We say that $G$ has the \emph{weakened (FA) property, (FA$^-$)} if for every action of $G$ on some tree $T$ we either have that $G$ acts elliptically on $T$ or that the action of $G$ on $T$ fixes some bi-infinite geodesic $L\subseteq T$ setwise.
\end{definition}
\subsection{Acylindrically Arboreal Groups}\label{sec:AAandAE}
In this section we formally define acylindrical arboreality and obtain some initial results.
\begin{definition} \label{def:AA} We say that a group $G$ is \emph{acylindrically arboreal} if $G$ acts acylindrically non-elementarily on some tree $T$. This action will give rise to a quotient graph of groups decomposition of $G$, which we will call a \emph{non-elementary acylindrical} splitting of $G$. 
\end{definition}

Any acylindrically arboreal group is acylindrically hyperbolic, as any simplicial tree is a $0$-hyperbolic metric space, but working with a simplicial tree allows us to use a much more combinatorial definition.

\begin{definition}\cite[Introduction]{Weidmann07}\label{Def:kc} Let $G$ be a group acting on some tree $T$ and let $k\geq0$ and $C>0$ be integers. We say that the action of $G$ on $T$ is $(k, C)$\emph{-acylindrical} if the pointwise stabiliser of any edge path in $T$ of length at least $k$ contains at most $C$ elements.
\end{definition}

The following theorem is essentially due to Minasyan and Osin \cite{minasyanOsin13}. We include a brief proof demonstrating how to use the cited lemma to obtain this result.
\begin{theorem}\cite[Lemma~4.2]{minasyanOsin13}\label{Thm:KCisAA} Let $G$ be a group acting by isometries on a tree $T$. This action is acylindrical (in the sense of Definition~\ref{def:acyl}) if and only if there exist constants $k\geq 0$ and $C\geq 1$ such that the action of $G$ on $T$ is $(k, C)$-acylindrical.
\end{theorem}
\begin{proof}
    As in the statement, let $G$ be a group acting by isometries on a tree $T$. 

First assume that this action is acylindrical with constants $R(\epsilon)$ and $N(\epsilon)$ as in Definition~\ref{def:acyl}. The action of $G$ on $T$ must be $(\lceil R(0)\rceil, N(0))$-acylindrical by definition of acylindricity.

For the other direction assume that the action of $G$ on $T$ is $(k, C)$-acylindrical for some constants $k$ and $C$. Let $\epsilon>0$, and set \[R(\epsilon)=k+2\epsilon+6\text{, }N(\epsilon)=2(2\epsilon+1)C.\]

Let $x, y\in T$ such that $d_T(x, y)\geq R(\epsilon)$, and let $u$ and $v$ be the closest vertices to $x$ and $y$ respectively on the geodesic $[x, y]$ such that $u$ and $v$ are at distance at least $\epsilon+1$ from $x$ and $y$ respectively. Thus the distance from $x$ to $u$ is bounded above by the distance from $x$ to the nearest vertex on $[x, y]$ plus $\lceil1+\epsilon\rceil=1+\lceil\epsilon\rceil$, so 
\[1+\epsilon\leq d_T(u, x)\leq 1+1+\lceil\epsilon\rceil\leq 3+\epsilon.\]
The distance $d_T(v, y)$ is similarly bounded, and by the triangle inequality \[d_T(u, y)\geq d_T(x, y)-d_T(u, x)\geq k+\epsilon+5>\epsilon.\] Similarly $d_T(v, x)>\epsilon$, and again by the triangle inequality we have that
\[d_T(u, v)\geq d_T(x, y)-d_T(u, x)-d_T(v, y)\geq k,\]
so it follows that $d_T(u, v)\geq k$ and $d_T(\{u, v\}, \{x, y\}) >\epsilon,$
and we can invoke \cite[Lemma~4.2]{minasyanOsin13} to see that the pointwise quasi-stabiliser $\Pqs_G^\epsilon(\{x, y\})$ is contained in at most $2(2\epsilon+1)$ cosets of $\Ps_G(\{u, v\})$, the pointwise stabiliser of $\{u, v\}$. By construction the geodesic path between $u$ and $v$ must use at least $k$ edges and so $|\Ps_G(\{u, v\})|\leq C$ by definition of definition of $k$ and $C$. It follows that $|\Pqs_G^\epsilon(\{x, y\})|\leq 2(2\epsilon+1)C=N(\epsilon)$. Thus our action was acylindrical with constants $R(\epsilon)$ and $N(\epsilon)$.
\end{proof}
The class of acylindrically arboreal groups has the following set of inheritance properties, claim (3) of which follows immediately from a restriction of \cite[Lemma~7.1]{Osin13} to actions on trees.
\begin{lemma}\label{lem:inheritance}
    Let $G$ be an acylindrically arboreal group. Then the following hold.
    \begin{enumerate}
        \item Any extension of $G$ along a finite kernel must be acylindrically arboreal.
        \item Any quotient of $G$ by a finite normal subgroup must be acylindrically arboreal.
        \item \cite[Lemma~7.1]{Osin13} Let $H\leq G$ be \emph{s-normal}, i.e. for all $g\in G$, $gHg^{-1}\cap H$ is infinite. Then $H$ is acylindrically arboreal. In particular, any finite index subgroup of $G$ is acylindrically arboreal.
    \end{enumerate}
\end{lemma}
\begin{proof}
    We prove each claim separately. Let $G$ be an acylindrically arboreal group, so $G$ admits a non-elementary acylindrical action on a tree $T$.
    \begin{enumerate}
        \item Let $E$ be an extension of $G$ such that the natural projection $E\rightarrow G$ has finite kernel $K$. Then $E$ acts on $T$ through this quotient, and given that by Theorem~\ref{Thm:KCisAA} the action of $G$ on $T$ was $(k, C)$-acylindrical for some $k$ and $C$ we must have that the action of $E$ on $T$ is $(k, C|K|)$-acylindrical, and thus acylindrical in the sense of Definition~\ref{def:acyl} by Theorem~\ref{Thm:KCisAA}. The preimage of any loxodromic element in $G$ must be a loxodromic element in $E$ so the action of $E$ on $T$ is not elliptic, and $E$ cannot be virtually cyclic as $G$ was not virtually cyclic. Thus the action of $E$ on $T$ must be non-elementary by Theorem~\ref{Thm:categorisedActions}, so $E$ is acylindrically arboreal as claimed.
        \item Let $K$ be a finite normal subgroup of $G$. Then the action of $K$ on $T$ must have a fixed point $x$ by Lemma~\ref{lem:elipticWellDef}. Let $T'$ be the convex hull of the orbit of $x$ in $T$ under the action of $G$, on which $G$ still acts non-elementarily acylindrically. Then by the normality of $K$ it must fix the entire orbit of $x$ pointwise, and hence fix all of $T'$. The action of $G$ on $T'$ must therefore factor through the quotient $G/K$, and thus $G/K$ is acylindrically arboreal as claimed.
        \item This claim follows immediately from a restriction of the proof of \cite[Lemma~7.1]{Osin13} to the case of acylindrically arboreal groups.\qedhere 
    \end{enumerate}
\end{proof}
Further, the following lemma is simply a restriction of a result of Osin to the case of acylindrically arboreal groups.
\begin{lemma}\cite[Corollary~7.2]{Osin13}\label{lem:BrokenProducts} Let $G$ be a group acting acylindrically on a tree $T$, and let $H\cong H_1\times H_2$ be a subgroup of $G$ that decomposes as the direct product of two infinite groups. Then $H$ must act elliptically on $T$.
\end{lemma}

Finally, using the combinatorial definition of acylindricity we prove the following useful result.
\begin{lemma}\label{lem:AAimpliesFinite}
	Let $G$ be a finitely generated acylindrically arboreal group. Then there exists a non-elementary acylindrical splitting $(\Gamma, \mathfrak{G})$ of $G$ where $\Gamma$ has exactly one edge.
\end{lemma}
\begin{proof}
    Let $(\Gamma', \mathfrak{G}')$ be a non-elementary $(k, C)$-acylindrical splitting of $G$ with Bass--Serre tree $T'=T(\Gamma', \mathfrak{G}')$, and let $e\in E(T')$ be an edge in $T'$. We may assume that no proper subtree of $T'$ is fixed setwise by $G$. Let $G\cdot e$ be the $G$-orbit of $e$ in $T'$, and define a tree $T$ with 
    \[V(T)=\{v_U\colon U\subseteq (T'-(G\cdot e))\text{ is a connected component}\}, \] and an edge between any two vertices whose labels are joined by an edge of $G\cdot e$. 
    
    Then $G$ acts on $T$ with one orbit of edges, and there is a natural $G$-equivariant map $\phi:T'\rightarrow T$ which is a bijection when restricted to the interior of $G\cdot e$, and which is Lipschitz, i.e. for all $u, v\in V(T')$, $d_T(\phi(u), \phi(v))\leq d_{T'}(u, v)$. It thus follows from equivariance and the bijectivity of the restriction of $\phi$ above that the action of $G$ on $T$ is $(k, C)$-acylindrical.

    It only remains to show that this action is non-elementary. Since $G$ cannot be virtually cyclic by assumption we need only show that the action of $G$ on $T$ has no fixed point, but if the action of $G$ on $T$ were to have a fixed vertex $v_U$ then the action on $T'$ must fix the subtree $U\subseteq T'$ setwise contradicting our assumption on $T'$. The action of $G$ on $T$ is therefore non-elementary acylindrical as required.
\end{proof}

%% file: graphProducts.tex
\section{Graph Products of Groups}\label{Sec:GP}
In this section we prove Theorem~\ref{Thm:mainIntro} and explore its consequences. We recall some definitions.
\begin{definition}
    Let $\Gamma=(V, E)$ be a finite simple graph, and let $\mathcal{G}=\{G_v\}_{v\in V}$ be a collection of groups enumerated by the vertex set of $\Gamma$. We define the \emph{graph product} $\mathbb{GP}(\Gamma, \mathcal{G})$ to be the group 
    $\langle \mathcal{G}\mid [G_u, G_v]$ for $\{u,v\}\in E\rangle$.
    Thus $\mathbb{GP}(\Gamma, \mathcal{G})$ is the quotient of the free product of the groups $\mathcal{G}$ by the normal closure of the commutators of those pairs of groups whose labels appear in the edge set of $\Gamma$.

    We call each group $G_v\in\mathcal{G}$ the \emph{vertex group} of $v$. We call a graph product of groups $\mathbb{GP}(\Gamma, \mathcal{G})$ \emph{degenerate} if it has any trivial vertex groups or if $\Gamma$ has only one vertex, and we say that $\mathbb{GP}(\Gamma, \mathcal{G})$ is \emph{non-degenerate} otherwise.
\end{definition}
If $\Gamma$ is complete we recover the direct product of the vertex groups, and if $\Gamma$ is discrete we recover the free product of the vertex groups. The graph product can therefore be considered as a generalisation of these two concepts.
\begin{example}
    Some of the most well studied examples of graph products are \emph{Right Angled Artin Groups} (RAAGs), where every vertex group is a copy of $\mathbb{Z}$, and \emph{Right Angled Coxeter Groups} (RACGs), where every vertex group is a copy of $\Z/2\Z$.
\end{example}

We will use a concept of standard form for an element of a graph product $G=\mathbb{GP}(\Gamma, \mathcal{G})$ originally formulated by Green \cite{Green1990GraphPO}, although our discussion will follow that of Antol\'in and Minasyan \cite[Section~2]{Antolinminasyan11}.

For any element $g\in G$ we can write $g$ as a \emph{word} $W=(g_1,...,g_n)$, where $g=g_1\cdot\cdot\cdot g_n$ and each $g_i$ is an element of some $G_v\in\mathcal{G}$. We will call each $g_i$ a \emph{syllable} of the word $W$. We say that $n$ is the \emph{length} of the word $W$. We say that two syllables $g_i$ and $g_j$ of $W$ with $i<j$ can be \emph{joined together} if either syllable is $1$ or if there exists some $G_v\in\mathcal{G}$ such that $g_i, g_j\in G_v$ and for all $i<k<j$, $g_k\in G_{u_k}$ with $u_k\in N_\Gamma(v)$. In such a case $g_j$ commutes with all such $g_k$ in $G$, so $W$ represents the same element of $G$ as the word \[(g_1,...,g_{i-1},g_ig_j,g_{i+1},...,g_{j-1}, g_{j+1},...,g_n)\], whose length is strictly smaller.

A word $W=(g_1,..., g_n)$ is called \emph{reduced} if it is empty or if $g_i\neq 1$ for all $i$ and no two distinct syllables of $W$ can be joined together.

Let $W=(g_1,...,g_n)$ be a (not necessarily reduced) word in a graph product $G=\mathbb{GP}(\Gamma, \mathcal{G})$. For consecutive syllables $g_i\in G_u$, $g_{i+1}\in G_v$ with $\{u, v\}$ an edge of $\Gamma$, we can interchange $g_i$ and $g_{i+1}$. This is known as \emph{syllable shuffling}.

We will refer to the following result of Green without proof.
\begin{theorem}\textit{\cite[Theorem~3.9]{Green1990GraphPO}}\label{Thm:StandardGP}
    Let $G=\mathbb{GP}(\Gamma, \mathcal{G})$ be a graph product. Then every element of $G$ can be represented by a reduced word. Moreover, if two reduced words represent the same element of $G$ then one can be obtained from the other by applying a finite sequence of syllable shuffling.
\end{theorem}

Let $g\in G=\mathbb{GP}(\Gamma, \mathcal{G})$ and let $W=(g_1,...,g_n)$ be a reduced word representing $g$. We define the \emph{length} $|g|_\Gamma=n$ and the \emph{support} of $g$ to be \[\supp_\Gamma(g)=\{v\in V(\Gamma)\mid \text{there exists } i\in\{1,...,n\}\text{ such that }g_i\in G_v\}.\] For a subset $X\subset G$ we define $\supp_\Gamma(X)=\bigcup_{g\in X}\supp_\Gamma(g)$.

Finally, we define $\lv(g)$ and $\fv(g)$ to be the sets of all $u\in V(\Gamma)$ such that some reduced word for $g$ ends or begins with a syllable from $G_u$ respectively (here f and l stand for ``first'' and ``last''). All of these concepts are well defined by Theorem~\ref{Thm:StandardGP}.

\subsection{Full and Parabolic Subgroups of Graph Products of Groups}
We define a \emph{full} and a \emph{parabolic} subgroup of a graph of groups as follows.
\begin{definition} \label{def:ParaGP}
    Let $\mathbb{GP}(\Gamma, \mathcal{G})$ be a non-degenerate graph product of groups, and let $A\subset V(\Gamma)$. We define the \emph{full subgroup on $A$}, $G_A$, to be the subgroup of $G$ generated by the vertex groups associated to the vertices of $A$, and by convention define $G_\emptyset=\{1\}$. We say that a subgroup $H\leq G$ is \emph{parabolic} if it is conjugate to a full subgroup.
\end{definition}
\begin{remark} The full subgroup $G_A$ is always isomorphic to the graph product of groups given by the subgraph of $\Gamma$ induced by the vertex set $A$, with vertex groups inherited from $\Gamma$.
\end{remark}
This section will list  without proof a collection of results of Antol\'in, Minasyan and Osin \cite{Antolinminasyan11, minasyanOsin13} on the subject of full and parabolic subgroups that will provide us with the tools to study acylindrical splittings of graph products of groups.

\begin{lemma}\textit{\cite[Lemma~3.2]{Antolinminasyan11}}\label{lem:amLem3.2} Let $G=\mathbb{GP}(\Gamma,\mathcal{G})$ be a non-degenerate graph product of groups. Suppose $U\subseteq V(\Gamma)$ and $g, x, y$ are some elements of $G$ such that $x$ is non-trivial, $gxg^{-1}=y$, $\supp_\Gamma(y)\subseteq U$ and $\lv(g)\cap \supp_\Gamma(x)=\emptyset$. Then $g$ can be represented by a reduced word $(h_1,...,h_r, h_{r+1},...,h_n)$, where $h_1,..., h_r\in G_U$ and $h_{r+1},...,h_n\in G_{\lk_\Gamma(\supp_\Gamma(x))}$.
\end{lemma}
\begin{remark}
    This result holds for trivial $x$ by (somewhat counter-intuitively) defining the link of the empty set in a graph $\Gamma$ to be the entire vertex set. The result in such a case is completely trivial, so we have opted to exclude it so as to highlight the interesting case.
\end{remark}
\begin{lemma}\textit{\cite[Lemmas~3.3,~3.4, Corollary~3.8]{Antolinminasyan11}, \cite[Lemma~6.4]{minasyanOsin13}}\label{lem:fullSub} Suppose that $G=\mathbb{GP}(\Gamma, \mathcal{G})$ is a graph product of groups. Then the following hold.
\begin{enumerate}
    \item Let $\mathcal{S}$ be an arbitrary collection of subsets of $V(\Gamma)$. Then $\bigcap_{S\in\mathcal{S}} G_S=G_T$ for $T=\bigcap_{S\in \mathcal{S}} S\subseteq V(\Gamma)$.
    \item If $U, W\subset V(\Gamma)$ and $g\in G$ then there exists some subset $P\subseteq U\cap W$ and $h\in G_W$ such that $gG_Ug^{-1}\cap G_W=hG_Ph^{-1}$.
    \item If $U, W\subseteq V(\Gamma)$ and $g_1, g_2\in G$ are such that for all $u\in U$, $G_u$ is non-trivial and such that $g_1G_Ug_1^{-1}\leq g_2G_Wg_2^{-1}$ then $U\subseteq W$.
    \item If $\mathcal{P}$ is an at most countable collection of parabolic subgroups of $G$ then $\bigcap_{P\in\mathcal{P}}P$ must itself be parabolic.
\end{enumerate}
\end{lemma}
We use these lemmas to make the following definitions.
\begin{definition} \label{def:essential} Let $G=\mathbb{G}(\Gamma, \mathcal{G})$ be a non-degenerate graph product of groups. Suppose that $P=gG_Ug^{-1}$ with $U\subseteq V(\Gamma)$ and $g\in G$, so $P$ is a parabolic subgroup of $G$. We define the \emph{essential support} $\esupp_\Gamma(P)$ to be the set $U$. This is well-defined Lemma~\ref{lem:fullSub}(3).

Given an arbitrary subset $X\subseteq G$ we define the \emph{parabolic closure} $\Pc(X)$ to be the intersection of all parabolic subgroups containing $X$, which is itself parabolic by Lemma~\ref{lem:fullSub}(4). We extend the definition of essential support by defining $\esupp_\Gamma(X)=\esupp(\Pc(X))$ for all $X\subseteq G$.
\end{definition}

The following definition will be crucial to the proof of Theorem~\ref{Thm:mainIntro}.
\begin{definition} Let $G=\mathbb{G}(\Gamma, \mathcal{G})$ be a non-degenerate graph product of groups.
We say that a pair of vertices $u$ and $v$ of $\Gamma$ are \emph{separated} (with respect to the graph product $\mathbb{GP}(\Gamma, \mathcal{G})$) if the edge distance between $u$ and $v$ is at least $2$ and the full subgroup $G_{\lk_\Gamma(\{u, v\})}$ is finite.
\end{definition}
\begin{example} \label{ex:diam3}
    If $G=\mathbb{G}(\Gamma, \mathcal{G})$ is a graph product of groups with $\diam(\Gamma)\geq 3$ then we will always have a pair of separated vertices. Indeed, if $a, b\in V(\Gamma)$ are edge distance three apart then $\lk_\Gamma(\{a, b\})$ must be empty, and so induces a trivial full subgroup.
\end{example}

\subsection{Acylindrical Hyperbolicity of Graph Products}\label{sec:acylGP}
In this section we recall results of Minasyan and Osin which provide a condition for a graph product of groups to be acylindrically hyperbolic. Recall that a graph $\Gamma$ is \emph{irreducible} if its graph theoretical complement --- the graph obtained from $\Gamma$ by replacing every edge with a non-edge and every non-edge with an edge --- is connected.

\begin{theorem}\label{thm:OsinGPAH}\textit{\cite[Theorem~2.12]{minasyanOsin13}}
     Let $G = \mathbb{GP}(\Gamma, \mathcal{G})$ be a non-degenerate graph product of groups with $\Gamma$ irreducible. Suppose that $H\leq G$ is a subgroup that is not contained in a proper parabolic subgroup of $G$. Then $H$ is either virtually cyclic or acylindrically hyperbolic.
\end{theorem}
The following corollary is then immediate,.
\begin{corollary}\cite[Corollary~2.13]{minasyanOsin13}\label{cor:OsinGPAH}
    Let $G = \mathbb{GP}(\Gamma, \mathcal{G})$ be a non-degenerate graph product of groups with $\Gamma$ irreducible. Then $G$ is either virtually cyclic or acylindrically hyperbolic.
\end{corollary}
 It is also a clear corollary of Theorem~\ref{thm:OsinGPAH} that if $H$ is a subgroup of a non-trivial graph product of groups $\mathbb{GP}(\Gamma, \mathcal{G})$ and if the essential support of $H$ induces an irreducible subgraph of $\Gamma$ we have that $H$ is either acylindrically hyperbolic or virtually cyclic.
\subsection{Proof of Theorem~\ref{Thm:mainIntro}}
We restate the theorem here for clarity.
\introGPthm*
\begin{proof}
    Let $G$, $\Gamma$ be as in the statement. For the if direction, assume that there exists a pair of vertices $a, b\in V(\Gamma)$ that are separated and that $G$ is not virtually cyclic. Let $N=N_\Gamma(a)\cap N_\Gamma(b)=\lk_\Gamma(\{a, b\})$, which must induce a finite full subgroup $G_N$ of $G$ by choice of $a$ and $b$. Let $A=\Gamma-b$, $B=\Gamma-a$, and let $C=A\cap B$, so that $G\cong G_A*_{G_C} G_B$. We claim that the action of $G$ on the Bass--Serre tree of this splitting is $(3, |G_N|)$-acylindrical.
    
    Let $T$ be the Bass--Serre tree associated to this splitting, and let $P$ be a path in $T$ with three edges. We may assume without loss of generality that the middle edge is labelled $G_C$, so there exist $g\in G_A-G_C$, $h\in G_B-G_C$ such that the other two edges in $P$ are labelled $gG_C$ and $hG_C$, as shown in Figure~\ref{linefig}.
    
    \begin{figure} 
        \centering
        \begin{tikzpicture}[->,shorten >=1pt,auto,node distance=3cm,
                thick]

            \node (1) at (-1, 0) {$G_A$};
            \node (2) at (1, 0) {$G_B$};
            \node (3) at (-2.5, -1) {$gG_B$};
            \node (4) at (2.5, 1) {$hG_A$};

            \path[-]
            (1) edge node {$G_C$} (2)
                edge node {$gG_C$} (3)
            (2) edge node {$hG_C$} (4);
        \end{tikzpicture}
        \caption{\label{linefig}A generic 3-path in the Bass--Serre tree of the amalgam $G_A*_{G_C}\nolinebreak G_B$ can be assumed to use $G_c$ as its middle edge as the action is edge-transitive and by isometries.}
    \end{figure}
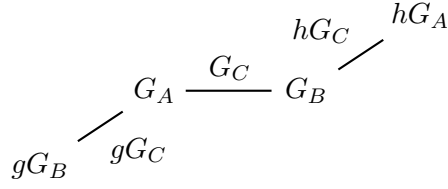    
    We first consider the subgroup $gG_Cg^{-1}\cap G_C$. If $l_{V_{\Gamma}}(g)$ contains any vertex of $C$ we can replace $g$ with $g'$ such that $|g'|_\Gamma<|g|_\Gamma$ and such that $gG_C=g'G_C$, a process which will terminate by the assumption that $g\notin G_c$. We can thus assume without loss of generality that $\lv(g)\cap C=\emptyset$. 

    If $gG_Cg^{-1}\cap G_C$ contains only the trivial element then $\Ps_G(P)$ is trivial and we are done, so assume that there exists $y\in gG_Cg^{-1}\cap G_C$ that is non-trivial. Then $y\in G_C$ and $y=gxg^{-1}$ for some non-trivial $x\in G_C$. Since $x$ is supported in $C$ we have that $\supp_\Gamma(x)\cap \lv(g)=\emptyset$ and we can invoke Lemma~\ref{lem:amLem3.2} with $U=C$, which shows that $g$ is supported in $C\cup \lk_\Gamma(\supp_\Gamma(x))$. Thus if $\supp_\Gamma(x)$ contains a vertex outside of $N_\Gamma(a)$ then $g$ is supported in $B$, contradicting the fact that $g\in G_A-G_C$ which is disjoint from $G_B$. It follows that $x$ is supported in $N_\Gamma(a)$, and so $y\in gG_{N_\Gamma(a)}g^{-1}$ meaning that
    \[gG_Cg^{-1}\cap G_C \subseteq gG_{N_\Gamma(a)}g^{-1}.\]
    Similarly
    \[ hG_Ch^{-1}\cap G_C \subseteq hG_{N_\Gamma(b)}h^{-1},\]
    and thus
    \[\Ps_G(P)=gG_Cg^{-1}\cap G_C\cap hG_Ch^{-1}=\left(gG_Cg^{-1}\cap G_C\right)\cap\left(hG_Ch^{-1}\cap G_C\right)\]\[\subseteq gG_{N_\Gamma(a)}g^{-1}\cap hG_{N_\Gamma(b)}h^{-1}.\]
    This last set is conjugate to $(h^{-1}gG_{N(a)}g^{-1}h)\cap G_{N(b)}$, and by Lemma~\ref{lem:fullSub}~(2) there exist some $f\in G$ and $N'\subseteq N$  such that 
    \[fG_{N'}f^{-1} =h^{-1}gG_{N(a)}g^{-1}h\cap G_{N(b)}.\]
    It follows that 
    $|\Ps_G(P)|\leq |fG_{N'}f^{-1}|\leq |G_N|,$
    which is finite by assumption. Thus the action of $G$ on the Bass--Serre tree $T$ given by $G_A*_{G_C}G_B$ must be $(3, |G_N|)$-acylindrical, and so must be acylindrical by Theorem~\ref{Thm:KCisAA}.

    It remains to show that the action of $G$ on $T$ is non-elementary. The subgroup $G_C$ is not equal to $G_A$ or $G_B$ as $\mathbb{GP}(\Gamma, \mathcal{G})$ is non-degenerate, so the splitting $G=G_A*_{G_C}G_B$ is non-trivial and the action of $G$ on $T$ cannot be elliptic. The group $G$ is not virtually cyclic by assumption, and so it follows that $G$ is an acylindrically arboreal group as required. 

    \bigskip For the only if direction, being virtually cyclic clearly prohibits $G$ from being acylindrically arboreal by definition of non-elementary, so assume that $\Gamma$ contains no pair of separated vertices. We will show in such a case that if $G$ acts acylindrically on some tree $T$ then this action must be elliptic. 
    
    By Example~\ref{ex:diam3} we must then have that $\diam(\Gamma)<3$, so by assumption we have that $\diam(\Gamma)=2$. Assume that $G$ acts acylindrically on some tree $T$, and fix the generating set $S$ of $G$ to be the union of a collection of generating sets of its vertex groups. Let $a, b\in V(\Gamma)$ be distinct vertices.
    We separate into two cases based on whether or not $a$ and $b$ lie in an induced $P_3$ subgraph of $\Gamma$.

    If $a$ and $b$ lie in an induced $P_3$ then, if $u$ and $v$ are the endpoints of this $P_3$, the full subgroup $G_{\{a, b\}}$ is entirely contained in the full subgroup $G_{\Lambda_{u,v}}$, where $\Lambda_{u,v}:=\{u, v\}\cup \lk_\Gamma(\{u, v\})$. The group $G_{\Lambda_{u,v}}$ is the direct product of $G_{\lk_\Gamma(\{u, v\})}$ and $G_{\{u, v\}}$, both of which are infinite as $u$ and $v$ are not neighbours in $\Gamma$ by construction and $u$ and $v$ are not separated. Thus by Lemma~\ref{lem:BrokenProducts} $G_{\Lambda_{u,v}}$ must act elliptically, and so must $G_{\{a, b\}}$ as a subgroup of $G_{\Lambda_{u,v}}$.

    If $a$ and $b$ do not lie in an induced $P_3$ they must then be adjacent by the fact that $\diam(\Gamma)=2$, so we have three subcases.
    \begin{enumerate}
        \item Both $G_a$ and $G_b$ are finite, so $G_{\{a, b\}}$ is finite and must act elliptically.
        \item Both $G_a$ and $G_b$ are infinite, so $G_{\{a, b\}}$ is the product of two infinite groups and thus acts elliptically by Lemma~\ref{lem:BrokenProducts} as above.
        \item Exactly one of $G_a$ and $G_b$ is infinite. Without loss of generality assume that $G_a$ is infinite and $G_b$ is finite. Assume for contradiction that $G_{\{a, b\}}$ contains a loxodromic element $g$. We claim that this implies that $\Gamma$ is a complete graph.
    
        We have that $\lk_\Gamma(a)=\lk_\Gamma(b)$ to avoid inducing a $P_3$ containing $a$ and $b$. Similar to above let $\Lambda_{a,b}=\{a, b\}\cup \lk_\Gamma(\{a, b\})$. The group $G_{\Lambda_{a, b}}$ is the direct product of $G_{\lk_\Gamma(\{a, b\})}$ and $G_{\{a, b\}}$, so $G_{\lk_\Gamma(\{a, b\})}$ must be finite by Lemma~\ref{lem:BrokenProducts} as $G_a\leq G_{\{a, b\}}$ is infinite and we are assuming that $G_{\{a, b\}}$ contains a loxodromic element. Thus $\lk_\Gamma(\{a, b\})$ must either be a complete graph with finite vertex groups or the empty graph as the graph product over any non-complete graph is infinite, and so the subgraph induced by $\Lambda_{a,b}$ is complete. 
    
        Assume for contradiction that there exists $v\in V(\Gamma)-\Lambda_{a,b}$. We may assume without loss of generality that $v$ is neighbour to some element of $\Lambda_{a,b}$ by the fact that $\diam(\Gamma)= 2$. If $v$ is a neighbour of one of $a$ or $b$, then it must also be a neighbour of the other to avoid inducing a $P_3$ on $v, a, b$, and so $v\in\lk_\Gamma(\{a, b\})$ contradicting our choice of $v$. It follows that $v$ is a neighbour of some element of $\lk_\Gamma(\{a, b\})$. However, we then have that $a, v$ is a separated pair as $\lk_\Gamma(\{a, v\})\subseteq \lk_\Gamma(a)\subseteq \lk_\Gamma(\{a, b\})$ which corresponds to a finite full subgroup, providing the desired contradiction. It follows that $V(\Gamma)=\Lambda_{a,b}$, and $\Gamma$ is complete as claimed. 
        
        This contradicts the assumption that $\diam(\Gamma)\geq2$, so we find that no such loxodromic element $g$ exists and $G_{\{a, b\}}$ must act elliptically.
    \end{enumerate}
    Now let $w\in V(\Gamma)$ be a vertex. Since $\diam(\Gamma)= 2$ there exists some vertex $x\in V$ distinct from $w$. The full subgroup $G_{\{w, x\}}$ then embeds in an elliptic subgroup of $G$ as above, so is itself elliptic, and it follows that $G_w$ must act elliptically on $T$.
    
    Finally, let $g_1, g_2\in S$ be generators contained in the vertex groups $G_{v_1}$ and $G_{v_2}$ respectively. The full subgroup $G_{\{v_1, v_2\}}$ must act elliptically on $T$ as above, and so $g_1$ and $g_2$ must share a fixed point. All conditions of Corollary~\ref{cor:eliptgen} are therefore satisfied,  so the action $G$ on $T$ contains no loxodromic elements and must be elliptic by Theorem~\ref{Thm:categorisedActions}.

    It follows that every acylindrical action of $G$ of a tree $T$ must be elliptic, so $G$ is not acylindrically arboreal as required.
\end{proof}
\subsection{Consequences of Theorem~\ref{Thm:mainIntro}}
Our first corollary is the following natural restriction to the case where all vertex groups are infinite.

\begin{corollary}\label{cor:GPinfVertex}
    Let $G=\mathbb{GP}(\Gamma, \mathcal{G})$ be a non-degenerate graph product of groups with infinite vertex groups. Then we have that $G$ is acylindrically arboreal if and only if $\diam(\Gamma)\geq 3$.
\end{corollary}
\begin{proof}
    For the if direction the diameter being at least three implies the existence of a pair of separated vertices $\{a, b\}$ by Example~\ref{ex:diam3}. The groups $G_a$ and $G_b$ are both infinite, so $G$ is not virtually cyclic and $G$ is acylindrically arboreal by Theorem~\ref{Thm:mainIntro}.
    
    For the only if direction, assume that $\diam(\Gamma)\leq2$. If $\diam(\Gamma)=1$ then $G$ is the direct product of infinite groups, which is not acylindrically arboreal by Lemma~\ref{lem:BrokenProducts}. If $\diam(\Gamma)=2$ then any pair of non-adjacent vertices share a neighbour with an infinite vertex group, and thus no pair of vertices of $\Gamma$ can be separated. The group $G$ therefore cannot be acylindrically arboreal by Theorem~\ref{Thm:mainIntro} as required.
\end{proof}
We can now prove Proposition~\ref{prop:brokenRAAG}, which we restate here for clarity.
\brokenraag*
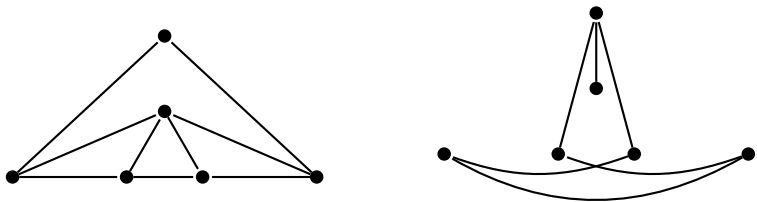
\begin{figure} \label{Fig:CoolRAAG}
        \centering
    \begin{minipage}{0.45\textwidth}
    \begin{flushright}
    \begin{tikzpicture}[->,shorten >=1pt,auto,node distance=3cm,
            thick]

        \node[draw, circle, fill, inner sep=1.5pt] (a) at (0, 0) {};
        \node[draw, circle, fill, inner sep=1.5pt] (b) at (-0.5, -0.87) {};
        \node[draw, circle, fill, inner sep=1.5pt] (c) at (0.5, -0.87) {};
        \node[draw, circle, fill, inner sep=1.5pt] (d) at (0, 1) {};
        \node[draw, circle, fill, inner sep=1.5pt] (e) at (-2, -0.87) {};
        \node[draw, circle, fill, inner sep=1.5pt] (f) at (2, -0.87) {};
            
         \path[-]
          (b) edge node {} (a)
              edge node {} (c)
          (a) edge node {} (c)
          (e) edge node {} (d)
              edge node {} (a)
              edge node {} (b)
          (f) edge node {} (d)
              edge node {} (a)
              edge node {} (c);
          
    \end{tikzpicture}
    \end{flushright}
    \end{minipage}\hfill
     \begin{minipage}{0.45\textwidth}
     \begin{flushleft}
     \begin{tikzpicture}[->,shorten >=1pt,auto,node distance=3cm,
            thick]

        \node[draw, circle, fill, inner sep=1.5pt] (a) at (0, 0) {};
        \node[draw, circle, fill, inner sep=1.5pt] (b) at (-0.5, -0.87) {};
        \node[draw, circle, fill, inner sep=1.5pt] (c) at (0.5, -0.87) {};
        \node[draw, circle, fill, inner sep=1.5pt] (d) at (0, 1) {};
        \node[draw, circle, fill, inner sep=1.5pt] (e) at (-2, -0.87) {};
        \node[draw, circle, fill, inner sep=1.5pt] (f) at (2, -0.87) {};

        \path[-]
        (b) edge node {} (d)
        (c) edge node {} (d)
        (a) edge node {} (d);

        \path[-]
        (f) [bend left=30] edge node {} (e)
        (f) [bend left=20] edge node {} (b)
        (c) [bend left=20] edge node {} (e);

    \end{tikzpicture}
    \end{flushleft}
     \end{minipage}
    \caption{\label{weirdgph}The underlying graph $\Gamma$ of the right angled Artin group in Proposition~\ref{prop:brokenRAAG} (left) and its graph theoretical complement (right).}
\end{figure}

\begin{proof}
    Let $G$ be the RAAG on the graph $\Gamma$ shown to the left of Figure 2. First we observe that $G$ is not acylindrically arboreal by Corollary~\ref{cor:GPinfVertex} as $\diam(\Gamma)=2$. 
    
    Next we show that $G$ is acylindrically hyperbolic, and by Corollary~\ref{cor:OsinGPAH} it suffices to show that $\Gamma$ is irreducible and that $G$ is not virtually cyclic. The second of these conditions is clear as $G$ has a $\Z^2$ subgroup, and for the first,  the graph theoretical complement of $\Gamma$ is the connected graph shown in Figure~\ref{weirdgph} (right), so $\Gamma$ is irreducible and $G$ is acylindrically hyperbolic as claimed.

    Finally we  show that $G$ does not have property (FA$^-$), and in fact has no infinite normal subgroups with (FA$^-$). A result of Antol\'in and Minasyan asserts that any subgroup of a right angled Artin group is either free abelian of finite rank or projects onto a free group of rank two \cite[Corollary~1.6]{Antolinminasyan11}. Thus $G$ does not have (FA$^-$) as it is not free abelian. Further, any non-trivial (and thus infinite) normal subgroup of $G$ with property (FA$^-$) would have to be free abelian as a projection to the free group of rank two would induce an action on a tree with no fixed points or line. However, the fact that $G$ is acylindrically hyperbolic implies that $G$ can contain no infinite normal (or indeed even s-normal) abelian subgroup \cite[Lemma~7.1]{Osin13}. Thus $G$ has no normal subgroup with property (FA$^-$) as claimed, and so $G$ has all properties required.
\end{proof}
\begin{remark}
    This example is by no means unique - indeed, the RAAG on a $C_5$ graph would also satisfy all required properties.
\end{remark}

We can also consider subgroups of graph products using Theorem~\ref{Thm:mainIntro}. Recall that, for a graph product of groups $G=\mathbb{GP}(\Gamma, \mathcal{G})$ and a subgroup $H$ of $G$, the \emph{parabolic closure} $\Pc(H)$ is the intersection of all parabolic subgroups of $G$ that contain $H$, and the \emph{essential support} $\esupp_\Gamma(H)$ is the essential support of the parabolic closure of $H$. 
\begin{corollary} \label{cor:subgrpsGP}
    Let $G=\mathbb{GP}(\Gamma, \mathcal{G})$ be a non-degenerate graph product of groups, and let $H\leq G$ be a subgroup of $G$. If the graph product induced by the full subgroup on the essential support $\esupp_\Gamma(H)$ contains a pair of separated vertices then $H$ is either acylindrically arboreal or virtually cyclic.
\end{corollary}
\begin{proof}
    By a contrapositive argument. Let $H$ be a subgroup of $G$ that is not virtually cyclic and not acylindrically arboreal. We claim that the sub-graph product of groups induced by $\esupp_\Gamma(H)$ contains no separated vertices. 
    
    Indeed, assume that this is not the case and $\esupp_\Gamma(H)$ contains a separated pair $a, b$. We may assume up to conjugation that $\Pc(H)=G_{\esupp_\Gamma(H)}$. Then $\Pc(H)$ splits as $G_A*_{G_C}G_B$ for $a\in A$, $b\in B$ as in the proof of Theorem~\ref{Thm:mainIntro}. If $T$ is the corresponding Bass--Serre tree $H$ must act elliptically on $T$ by Theorem~\ref{Thm:categorisedActions}, so the action of $H$ has a fixed point by definition. It follows that $H$ must lie entirely in some conjugate of $G_A$ or $G_B$, which are both proper full subgroups of $G_{\esupp_\Gamma(H)}$. This contradicts the definition of parabolic closure and thus $\esupp_\Gamma(H)$ contains no such separated pair $a, b$.
\end{proof}
Unfortunately the converse does not hold in general. Once again using the simple case of right angled Artin groups we can construct the following counterexample.
\begin{example}
    Let $G$ be the  RAAG on a $P_3$ whose vertices groups are generated by the letters $a, b$ and $c$ so that $a$ and $c$ generate a copy of $F_2$. The subgroup $\langle ab, bc\rangle$ is isomorphic to $F_2$ so is acylindrically arboreal, but its essential support is $\Gamma$, which contains no separated pair of vertices.
\end{example}

%% file: manifolds.tex
\section{3-Manifold Groups}\label{sec:hyperbolic}
In this section we will consider the acylindrical arboreality of certain fundamental groups of 3-manifolds, and provide a proof of Theorem~\ref{thm:intromfd}. We will provide a geometric condition on the fundamental group of a compact and orientable hyperbolic 3-manifold with empty or incompressible toroidal boundary which is equivalent to acylindrical arboreality, where by \emph{hyperbolic} we mean carries a complete metric of constant negative curvature on its interior. For avoidance of doubt, all surface embeddings will be assumed to be proper, i.e. if $f:S\hookrightarrow M$ is an embedding of a surface into a $3$-manifold $M$ then $f^{-1}(\partial M)=\partial S$.

\subsection{Relative Quasi-Convexity and Malnormality}
Let $G$ be a group, and $H$ a subgroup of $G$. We say that $H$ is \emph{malnormal} in $G$ if for all $g\in G\setminus H$, $gHg^{-1}\cap H=\{1\}$. The concept of malnormality is very closely linked with acylindricity, and in the context of relatively hyperbolic groups this is very well studied, but it can be difficult to find malnormal subgroups. We will instead use the related concept of \emph{relative quasi-convexity}, for which we will use the definition due to Osin \cite[Definition~1.8]{Osin06}. For a more thorough treatment of relative hyperbolicity and relative quasi-convexity, see \cite{Hruska08}.

Let $G$ be a group hyperbolic relative to a collection of subgroups $\mathcal{P}$, which we will always assume to be closed under conjugation. We say that a subgroup $H$ of $G$ is \emph{peripheral} if $H$ is conjugate into a peripheral subgroup $P\in\mathcal{P}$, and that an element $g\in G$ is a \emph{peripheral} element if $g$ is an element of some peripheral subgroup.

Intuitively, a subgroup of a group $G$ hyperbolic relative to a collection of subgroups $\mathcal{P}$ is relatively quasi-convex if the inclusion of $H$ into the Cayley graph of $G$ is close to being a convex subset once we have collapsed subsets corresponding to the peripheral subgroups. More formally, we have the following definition.
\begin{definition}
   Let $G$ be a group hyperbolic relative to a collection of subgroups $\mathcal{P}$ such that $G$ is generated by the finite set $X$. We say that a subgroup $H$ of $G$ is \emph{relatively quasi-convex} with respect to $\mathcal{P}$, or quasi-convex relative to $\mathcal{P}$, if there exists some constant $\epsilon>0$ such that the following condition holds. Let $f$ and $g$ be elements of $H$, and $p$ a geodesic between $f$ and $g$ in the Cayley graph of $G$ with respect to the generating set $X\cup\mathcal{P}$. Then for any vertex $w\in p$ there exists a vertex $v\in H$ such that $d_{X\cup\mathcal{P}}(v, w)\leq\epsilon$.
\end{definition}

We have the following definitions and theorem that link relative quasi-convexity to malnormality.
\begin{definition}
    Let $G$ be a group and let $H$ be a subgroup of $G$. Let $g_1, ...,g_n\in G$ be a set of $n$ elements in $G$. If $g_1H,...,g_{n}H$ is a set of disjoint cosets of $H$ in $G$, we say that the conjugates $g_1 Hg_1^{-1}, ..., g_n Hg_n^{-1}$ of $H$ in $G$ are \emph{essentially distinct}.

    Now assume that $G$ is hyperbolic relative to a collection of subgroups $\mathcal{P}$. We say that a subgroup $H$ of $G$ has \emph{finite relative height (in $G$)} if there exists some constant $n_H\in\Z_{>0}$ such that for any set $g_1 Hg_1^{-1}, ..., g_{n_H} Hg_{n_H}^{-1}$ of $n_H$ essentially distinct conjugates of $H$ in $G$ we have that  $\displaystyle{\bigcap_{0<i\leq n_H}g_iHg_i^{-1}}$ is either finite or is a peripheral subgroup of $G$. 
\end{definition}
\begin{theorem} \cite[Theorem~1.4]{HW09} \label{thm:FinRelH}
    Let $G$ be a group hyperbolic relative to a collection of subgroups $\mathcal{P}$, and $H$ a relatively quasi-convex subgroup of $G$ with respect to $\mathcal{P}$. Then $H$ has finite relative height in $G$.
\end{theorem}
A special case of relative quasi-convexity occurs when $H$ is a subgroup of a hyperbolic group $G$ (i.e. $\mathcal{P}$ contains only the trivial subgroup). In this case we say that the subgroup in question is \emph{quasi-convex}. The restriction of Theorem~\ref{thm:FinRelH} to quasi-convex subgroups of hyperbolic groups was first proved in \cite[Main Theorem]{GitikMahanRipsSageev98}, and gives the stronger conclusion that the intersection of a certain number of cosets of $H$ must be finite, as we have no non-trivial peripheral elements. This implies that the action of a hyperbolic group on a tree with quasi-convex edge stabilisers and finitely many orbits of edges will be acylindrical, a fact that is well known among experts (see \cite{Linton22}, for example).

\subsection{Splittings of 3-Manifold Groups and Subgroup Tameness}
We will require some deep results from various authors in the field of 3-manifold groups. We begin with some important definitions.
\begin{definition}  Let $M$ be a compact and orientable 3-manifold and $S$ a compact surface embedded in $M$.
\begin{enumerate}
    \item We say that $S$ is \emph{boundary parallel} if there exists an isotopy of $S$ onto some boundary component of $M$.
    \item Assume $S\not\cong S^2$. A disc $D\subset M$ is a \emph{compressing disc} for $S$ if $S\cap D=\partial D$ and this intersection is transverse. A compressing disc is \emph{non-trivial} if $\partial D$ does not bound a disc in $S$. We say that $S$ is \emph{compressible} if it admits a non-trivial compressing disc, and we say that $S$ is \emph{incompressible} otherwise.

     By convention, a sphere $S^2$ embedded in a $3$-manifold $M$ is considered to be incompressible if it does not bound a ball, and a disc embedded in a $3$-manifold is incompressible if it is not homotopic into a boundary component. 
    \item  We say that $M$ is \emph{irreducible} if every sphere embedded in $M$ bounds a $3$-ball, or equivalently if $M$ contains no embedded incompressible sphere.
    \item  If $S$ is incompressible, we say that $S$ is \emph{$2$-sided} if there exists an embedding $h\colon S\times[-1, 1]\rightarrow M$ such that $h(x,0)=x$ for all $x\in S$.
\end{enumerate}
\end{definition}
 We then have the following lemma.
\begin{lemma}\cite[Corollary~6.2]{Hempel76}\label{lem:2sided}
    Let $M$ be a $3$-manifold and $S$ a $2$-sided incompressible surface in $M$. Then the natural map $i_*\colon\pi_1(S)\rightarrow\pi_1(M)$ induced by the inclusion $i\colon S\rightarrow M$ is injective.
\end{lemma} 

The main reason that incompressible surfaces are relevant to the results in this paper is the ability to construct such surfaces in three manifolds from splittings of the fundamental group. 

We say that a $3$-manifold $M$ is \emph{fibred} if it admits the structure of a surface bundle over $S^1$. In such a case there exists a $2$-sided incompressible surface $S\subset M$ such that $\pi_1(M)\cong\pi_1(S)\rtimes\mathbb{Z}$. We say that a subgroup $H\leq\pi_1(M)$ is a \emph{surface fibre subgroup} of $\pi_1(M)$ if it is finitely generated and if $M$ admits a surface bundle structure over $S^1$ with fibre surface $S$ such that $H$ is the image of $\pi_1(S)$ in $\pi_1(M)$.

We say that $H$ is a \emph{virtual surface fibre subgroup} of $\pi_1(M)$ if there exists some finite sheeted cover $M'$ of $M$ whose fundamental group contains $H$ as a fibre subgroup. A finite sheeted covering of $M$ corresponds to a finite index subgroup of $\pi_1(M)$, so $\pi_1(M')$ will be a finite index subgroup of $\pi_1(M)$ of which $H$ is a surface fibre subgroup.

We have a rigidity result of Stallings that we can use to detect whether a subgroup of the fundamental group of compact and irreducible 3-manifold is a surface fibre subgroup. This result is often called \emph{algebraic fibring}.

\begin{theorem}\textit{\cite[Theorems~1, 2]{Stallings1961OnFC}}\label{thm:algfibr}
    Let $M$ be a compact and irreducible 3-manifold and let $H\unlhd\pi_1(M)$ be a finitely generated normal subgroup. If \[\pi_1(M)/H\cong \Z\text{ and }H\not\cong \Z/2\Z,\] then $M$ can be expressed as a fibred 3-manifold with fibre surface $S$. Furthermore, we can choose $S$ such that $\pi_1(S)=H$
\end{theorem}

A similar result holds for amalgams in the case when $M$ is closed, due to Scott.

\begin{theorem}\textit{\cite[Theorem 2.3]{Scott72}}\label{thm:ScottAmalgams}
Let $M$ be a closed and orientable irreducible 3-manifold, and suppose $\pi_1(M)\cong A*_C B$ where $C\neq A$ or $B$ and $C$ is isomorphic to the fundamental group of a closed surface $S$. Then there is an incompressible embedding of $S$ in $M$ separating $M$ into $M_1$ and $M_2$ with $\pi_1(M_1)=A$, $\pi_1(M_2)=B$ and $\pi_1(S)=C$.
\end{theorem}

More generally, the following powerful result essentially due to Stallings, Epstein and Waldhausen can be used to construct 2-sided incompressible surfaces from any graph of groups decomposition of the fundamental group of a closed and orientable 3-manifold. We will use the version that appears in \cite{Culler1983VarietiesOG}.

\begin{theorem}\textit{\cite[Proposition~2.3.1]{Culler1983VarietiesOG}}\label{Thm:CullerShalen}
    Let $M$ be a compact orientable 3-manifold. For any non-trivial graph of groups decomposition $(\Gamma, \mathfrak{G})$ of $\pi_1(M)$ with finitely many edge groups there exists a non-empty system $\overline{\Sigma}=\{\Sigma_1,...,\Sigma_n\}$ of compact $2$-sided incompressible surfaces embedded in $M$, none of which are boundary parallel, such that for all $i$, $\Img(\pi_1(\Sigma_i)\rightarrow \pi_1(M))$ is contained in some edge group of $(\Gamma, \mathfrak{G})$ and for all connected components $M_j$ of $M\backslash\overline{\Sigma}$, $\Img(\pi_1(M_j)\rightarrow \pi_1(M))$ is contained in some vertex  group of $(\Gamma, \mathfrak{G})$.

    Moreover, if $\mathcal{K}\subset\partial M$ is a connected component such that $\Img(\pi_1(K)\rightarrow \pi_1(M))$ is contained in a vertex group of $(\Gamma, \mathfrak{G})$, then we may take all surfaces $\Sigma_i\in\overline{\Sigma}$ to be disjoint from $\mathcal{K}$.
\end{theorem}
\begin{remark}
    In the paper of Culler and Shalen, the authors do not conclude that the surfaces they construct are compact. However, compactness is an immediate consequence of their proof --- the surfaces are constructed as the connected components of the continuous preimage in the compact manifold of a closed set, so are themselves compact.
\end{remark}

If we assume that $M$ is hyperbolic then the fundamental group of $M$ will have powerful properties. In particular, many of these arise in part from the following powerful theorem.

\begin{theorem}\cite[(K.18)]{AschenbrennerFriedlWilton12}\label{thm:MrelHyp}
    Let $M$ be a compact and orientable hyperbolic $3$-manifold with empty or toroidal boundary. Then $\pi_1(M)$ is hyperbolic relative to the conjugacy classes of the subgroups arising from the boundary components of $M$. 
\end{theorem}

The following theorem and its corollary are well known and the arguments standard, although we include brief proofs for completeness.
\begin{theorem}\label{thm:ired}
    Let $M$ be a compact and orientable hyperbolic $3$-manifold with empty or toroidal boundary. Then $M$ is irreducible.
\end{theorem}
\begin{proof}
    Let $M$ be as in the statement. Then the universal cover of the interior of $M$ is a copy of $\mathbb{H}^3$, which has vanishing second homotopy group as it is contractible. It follows that the second homotopy group of the interior of $M$ must also vanish, and the boundary of $M$ contains no spherical components, so $M$ is irreducible as required as it can contain no non-trivial sphere.
\end{proof}
\begin{corollary}\label{cor:freeSplit}
    A compact and orientable hyperbolic $3$-manifold $M$ with empty or toroidal boundary that contains an embedded $2$-sided incompressible disc must be a solid torus. In particular, as a consequence of Theorem~\ref{Thm:CullerShalen}, $\pi_1(M)$ admits a non-trivial splitting over a trivial subgroup if and only if $M$ is a solid torus.
\end{corollary}
\begin{proof}
For the first part of this corollary, assume that $M$ contains an embedded disc , so there exists an embedding of a disc $D$ into $M$ such that the boundary of $D$ is a non-trivial curve embedded in some boundary component $T$ of $M$. We will argue as in \cite[Page 14, (3)]{HatcherNotes} that $M$ is a solid torus. 

Indeed, $\partial D$ is non-separating in $T$ by the assumption that it is non-trivial, so the surgery of $T$ along $D$ can be used to create a sphere in $M$ that must then bound a ball $B$ in $M$ as $M$ is irreducible by Theorem~\ref{thm:ired}. This ball must lie on the same side of $T$ as $D$, as $T$ is a boundary component of $M$, and so reversing the surgery on $T$ glues two discs in the boundary of $B$ together, creating a solid torus as required.

For the second part of this corollary, we first observe that if $M$ is a solid torus, then $\pi_1(M)=\Z$, so does admit a non-trivial splitting over the trivial subgroup given by the natural action on the real line. For the other direction, assume that $M$ admits a non-trivial splitting over the trivial subgroup. Then by Theorem~\ref{Thm:CullerShalen}, $M$ contains an embedded $2$-sided incompressible surface $S$ with trivial fundamental group. It then follows from the classification of surfaces that $S$ is a sphere or a disc, so must be a disc as $M$ is once again irreducible by Theorem~\ref{thm:ired}. Thus $M$ contains an embedded $2$-sided incompressible disc, so $M$ is a solid torus by the first part of this corollary.
\end{proof}

We have the following lemma for certain irreducible $3$-manifolds, that we will apply to compact hyperbolic $3$-manifolds with empty or toroidal boundary using Theorem~\ref{thm:ired}.

\begin{lemma}\cite[(C.2)]{AschenbrennerFriedlWilton12}\label{lem:tf}
    Let $M$ be an orientable, irreducible $3$-manifold with empty or toroidal boundary and infinite fundamental group. Then $\pi_1(M)$ is torsion free.
\end{lemma}
Many of our methods will rely on a powerful result known as \emph{subgroup tameness} to classify subgroups of the fundamental group of hyperbolic $3$-manifolds. For a full definition of the concept of geometric finiteness we direct the reader to the book \emph{3-Manifold Groups} by Aschenbrenner, Friedl and Wilton \cite[Chapter~5]{AschenbrennerFriedlWilton12}. For our purposes however, we will be able to define geometrically finite subgroups simply as relatively quasi-convex subgroups of hyperbolic $3$-manifold groups using the following theorem of Hruska.

\begin{theorem}\textit{\cite[Corollary~1.6]{Hruska08}} \label{thm:GFisQC}
	Let $G=\pi_1(M)$ be the fundamental group of a compact and orientable hyperbolic 3-manifold with empty or toroidal boundary and let $H$ be a finitely generated subgroup of G. Then $H$ is geometrically finite in $G$ if and only if $H$ is relatively quasi-convex with respect to the subgroups of $G$ that correspond to the fundamental groups of the boundary components.
\end{theorem}

We have the following remarkable and deep result due to various authors.
\begin{theorem}\textit{\cite[Theorem~5.2]{AschenbrennerFriedlWilton12}} \label{thm:tameness}
	Let $M$ be a hyperbolic 3-manifold and let $H\leq \pi_1(M)$ be a finitely generated subgroup. Then either
	\begin{enumerate}
		\item $H$ is a virtual surface fibre subgroup, or;
		\item $H$ is geometrically finite.
	\end{enumerate}
\end{theorem}

\subsubsection{$I$-bundles}
We will require some results about a special type of 3-manifold that is geometrically ``nice'' in terms of its fundamental group. We will briefly define the concept of an $I$-bundle over a surface, and include a strong result pertaining to $3$-manifolds whose boundary components induce finite index subgroups of the fundamental group. For proof and further exposition see \cite[ Chapter~10]{Hempel76}.

\begin{definition}
Let $S$ be a closed surface and let $I$ be the standard $[0,1]$ interval. An $I$-bundle over $S$ is a fibre bundle over $S$ with fibre $I$.
\end{definition}

\begin{theorem} \cite[Theorem~10.5]{Hempel76}\label{thm:HempelIBundles}
    Let $M$ be a compact 3-manifold which contains no 2-sided real projective plane and let $S$ be a compact, connected, incompressible surface embedded in $\partial M$ such that $S$ is not homeomorphic to the disc $B^2$, the sphere $S^2$ or the projective plane $P^2$. If the index $[\pi_1(M)\colon i_*\pi_1(S)]$ is finite then either:
        \begin{enumerate}
            \item $\pi_1(M)\cong \Z$ and $M$ is a solid torus or Klein bottle;
            \item $[\pi_1(M)\colon\pi_1(S)]=1$ and $M=S\times I$ with $S=S\times\{0\}$;
            \item $[\pi_1(M)\colon\pi_1(S)]=2$ and $M$ is an $I$-bundle over a compact surface $\overline{S}$, with $S$ a two sheeted cover of $\overline{S}$. We call such an $I$-bundle \emph{twisted}.
        \end{enumerate}
\end{theorem}
\begin{remark}
    All manifolds we consider in this paper will be orientable, and thus cannot contain 2-sided real projective planes due to the fact that, if $P^2$ is the real projective plane, $P^2\times[-1,1]$ is non-orientable. The condition that $M$ has no two-sided $P^2$ has however been included in the statement of Theorem~\ref{thm:HempelIBundles} for completeness.
\end{remark}

\subsection{A Geometric Condition for Acylindrical Arboreality}
In this section we will provide a proof of Theorem~\ref{thm:intromfd}.

We say that a manifold is \emph{closed} if it is compact and has empty boundary, and begin with the following lemma.

\begin{lemma} \label{lem:poisonVF}
    Let $M$ be a closed and orientable hyperbolic $3$-manifold and $H$ a virtual surface fibre subgroup of $G=\pi_1(M)$. If $A$ is a subgroup of $G$ that contains $H$ then either $A$ has finite index in $G$ or is itself a virtual surface fibre subgroup of $G$.
\end{lemma}
\begin{proof}
    Let $H$ be a non-trivial virtual surface fibre subgroup of $G$. Then by definition $H$ is finitely generated, and there exists $t\in G$ such that $F=H\rtimes \langle t\rangle$ is a finite index subgroup of $G$ corresponding to a finite sheeted cover of $M$ in which $H$ is a surface fibre subgroup. Let $A$ be a subgroup of $G$ that contains $H$. We separate into two cases. 
    
    \textbf{Case 1, $F\cap A=H$ : }In this case, $H$ has finite index in $A$ and is finitely generated, implying that $A$ is also finitely generated. Assume for contradiction that $A$ is not a virtual surface fibre subgroup of $G$. It then follows by Theorem~\ref{thm:tameness} that $A$ must be geometrically finite as it is finitely generated, and so by Theorems~\ref{thm:GFisQC} and~\ref{thm:FinRelH} we have that $A$ has finite relative height in $G$ with respect to the boundary subgroups of $G$, which are trivial as $M$ is closed. For all $0\neq j\in\Z$ we have that $t^j\notin A$, so the set of conjugates $\{t^j A t^{-j}\}_{j\in\mathbb{\Z}}$ is essentially distinct, and the intersection of all of these subgroup will contain $H$ as every power of $t$ fixes $H$ by conjugation. Thus, $H$ must be a peripheral or finite subgroup of $G$ by the fact that $A$ has finite relative height, and all of our peripheral subgroups are trivial in this case, so $H$ must be finite. The manifold $M$ is orientable with empty boundary and irreducible by Theorem~\ref{thm:ired}, and $G$ admits a non-trivial splitting so it must be infinite, so we may apply Lemma~\ref{lem:tf} to see that $G$ is torsion free, and $H$ in this case must be trivial. However, by Corollary~\ref{cor:freeSplit}, $\pi_1(M)$ admits a non-trivial splitting over the free subgroup if and only if $M$ is a solid torus, which contradicts the fact that $M$ is closed. It follows that $H$ cannot be trivial, leading to the desired contradiction. Thus, in the case where $A\cap F=H$ we must have that $A$ is a virtual surface fibre subgroup.
    
    \textbf{Case 2, $A\cap F$ contains $H$ as a proper subgroup : }Then there exist $0\neq l\in\Z$ and $h\in H$ such that $t^lh\in A\cap F$, but $h\in A\cap F$ by assumption that $A$ contains $H$, and so we must have that $t^l\in A\cap F$. It follows that $A\cap F$ contains the subgroup $\langle t^l, H\rangle$, which has finite index in $H\rtimes\langle t\rangle$, which in turn has finite index in $G$. Thus in this case $A$ contains the finite index subgroup $\langle t^l, H\rangle$ of $G$, and so $A$ must have finite index in $G$. Thus in both cases either $A$ is a virtual surface fibre subgroup of $G$ or $A$ has finite index in $G$, as required.
\end{proof}
We obtain the following result which categorises all closed and orientable hyperbolic 3-manifolds whose fundamental groups have a non-trivial splitting over a virtual surface fibre subgroup.

\begin{theorem}\label{thm:mfdCateg}
    Let $M$ be a closed and orientable hyperbolic $3$-manifold, and assume that $G=\pi_1(M)$ has a non-trivial graph of groups splitting over a virtual surface fibre subgroup $H$. Then either $M$ can be expressed as a surface bundle over the circle, and we can choose the fibre to have fundamental group equal to $H$, or $M$ is a pair of twisted $I$-bundles identified along their boundaries with $H$ equal to the subgroup of the fundamental group corresponding to the boundary component of either $I$-bundle. In particular, $H$ is normal in $G$ in both of these cases.
\end{theorem} 

\begin{proof}
    Let $H$ be a non-trivial virtual surface fibre subgroup of $G$. Then by definition $H$ is finitely generated and there exists $t\in G$ such that $F=H\rtimes \langle t\rangle$ is a finite index subgroup of $G$ corresponding to a finite sheeted cover of $M$ in which $H$ is a surface fibre subgroup.
    
    Now assume that $G$ splits non-trivially over $H$, so there exists a non-trivial graph of groups decomposition $(\Gamma, \mathfrak{G})$ of $G$ with one orbit of edges, whose stabilisers are conjugates of $H$. We split into two cases, based on whether $\Gamma$ has $1$ or $2$ vertices.

    \bigskip\textbf{Case 1, $\Gamma$ has 2 vertices : }In this case we can express $G$ as a non-trivial amalgam $A*_{H}B$ of two subgroups $A$ and $B$ along $H$. In this case, we claim that the subgroup $H$ must have finite index in both $A$ and $B$. Indeed, we first observe that as the amalgam decomposition $A*_HB$ is non-trivial by hypothesis on $H$, neither $A$ nor $B$ can have finite index in $G$. Therefore both $A$ and $B$ must be virtual surface fibre subgroups of $G$ by Lemma~\ref{lem:poisonVF}, so the normalisers $N_G(A)$ and $N_G(B)$ of $A$ and $B$ in $G$ must have finite index in $G$. By elementary properties of amalgamated free products and non-triviality of the amalgam $A*_HB$, the normaliser $N_G(A)$ of $A$ in $G$ must have that $N_G(A)\cap B=H$, so it follows from the fact that $N_G(A)$ has finite index in $G$ that $H=N_G(A)\cap B$ must also be finite index in $B$. Similarly $H$ will have finite index in $A$, and the claim is proved.
         
    The subgroup $H$ is a virtual surface fibre subgroup so it must be isomorphic to the fundamental group of a closed surface, and $M$ is irreducible by Theorem~\ref{thm:ired}, so we may apply Theorem~\ref{thm:ScottAmalgams} to realise $H$ geometrically as the fundamental group of some incompressible surface $\Sigma$ in $M$. The manifold $M$ cut along $\Sigma$ then has two components $M_1$, with $\pi_1(M_1)=A$, and $M_2$, with $\pi_1(M_2)=B$, both 3-manifolds whose fundamental groups are virtually $\pi_1(\Sigma)$. In particular, $\Sigma$ must be $2$-sided as it separates $M$, so by the remark after Theorem~\ref{thm:HempelIBundles} $\Sigma$ is not a real projective plane, and more generally neither $M_1$ nor $M_2$ can contain a $2$-sided projective plane by the same remark. Furthermore, if $\Sigma$ were a disc we would be able to conclude by Corollary~\ref{cor:freeSplit} that $M$ were a solid torus, contradicting the fact that $M$ is closed, and if $\Sigma$ were a sphere then Theorem~\ref{thm:ired} would imply that $\Sigma$ bounds a ball, contradicting incompressibility. Thus $\Sigma$ is not a disc, sphere, or real projective plane. Therefore, using again the fact that $M$ is irreducible and so each of $M_1$ and $M_2$ must be irreducible, we may invoke theorem~\ref{thm:HempelIBundles} on $M_1$ (and respectively on $M_2$) to see that one of the following occurs.
    
    \begin{enumerate}
        \item $\pi_1(M_1)\cong\Z$ and $M_1$ is a solid torus or Klein bottle. This however cannot occur, as in both of these cases the boundary fails to be $\pi_1$-injective, whereas we know by construction of $\Sigma$ that the inclusion of $\Sigma$ into $M$ (and therefore $M_1$ and $M_2$) is $\pi_1$-injective.
        \item $[\pi_1(M_1):\pi_1(\Sigma)]=1$, which cannot occur as our original splitting was non-trivial.
        \item $[\pi_1(M_1):\pi_1(\Sigma)]=2$ and $M_1$ is a twisted $I$-bundle with boundary $\Sigma$.
    \end{enumerate}
    
    Thus we must have that both $M_1$ and $M_2$ are twisted $I$-bundles with boundary $\Sigma$ as required. In this case, $H$ has index two in both $\pi_1(M_1)=A$ and $\pi_1(M_2)=B$, so is normal in both of these groups. Thus $H$ must be normal in $G=A*_HB$ as required.

    \textbf{Case 2, $\Gamma$ has exactly $1$ vertex : } Call this unique vertex $v$ $v$. We can then express $G$ as an HNN extension of $G_v$ along $H$, so call the stable letter of this extension $\gamma$. We claim that in this case the subgroup $G_v$ must be equal to $H$, and must be a genuine surface fibre subgroup of $G$, i.e. there exists some fibration of $M$ over the circle whose fibre is a surface with fundamental group $G_v=H$. Indeed, we must have that either $G_v$ is either finite index in $G$ or that $G_v$ is a virtual surface fibre subgroup of $G$ by Lemma~\ref{lem:poisonVF}. Since $G$ is an HNN extension with base group $G_v$, the subgroup $G_v$ cannot have finite index, and so we must in fact have that $G_v$ is a virtual surface fibre subgroup of $G$. It follows that there exists $t_v\in G$ such that $F_v=G_v\rtimes\langle t_v\rangle$ is a finite index subgroup of $G$ corresponding to a cover of $M$ in which $G_v$ is a surface fibre subgroup. 
    
    Thus there exists some integer $l>0$ such that $\gamma^l\in F_v$, and so both $\gamma^l$ and $\gamma^{-l}$ normalise $G_v$. However, by definition of $\gamma$ as the stable letter in the HNN extension of $G_v$ along $H$ and by elementary properties of HNN extensions of groups, this can only occur if $H=G_v$. It follows that $G=H\rtimes\langle \gamma\rangle$, $H$ is finitely generated by definition of a surface fibre subgroup. Furthermore, $G$ admits a non-trivial splitting and is therefore infinite, so is torsion free by Lemma~\ref{lem:tf}, so we have that $H\leq G$ is torsion free and thus $H\not\cong \Z/2\Z$. We may therefore invoke Theorem~\ref{thm:algfibr} (algebraic fibring) to see that $H=G_v$ is a surface fibre subgroup of $G$. The claim is therefore proved, and so in this case $H$ is a surface fibre subgroup so is normal by definition as required.
    
    It follows that either the splitting has two vertices, in which case we have proved that $M$ can be identified with a pair of twisted $I$-bundles identified along their boundaries with $H$ equal to the subgroup of the fundamental group corresponding to the boundary component of either $I$-bundle, or the splitting has one vertex, and $M$ is a fibred manifold where we can choose the fibre of $M$ to have fundamental group $H$ as required. Also as required, in both of these cases $H$ is normal in $G$.
\end{proof}

We can now prove Theorem~\ref{thm:mainHypMfd}, which we restate here.
\intromfdclass*

\begin{proof}
    First assume that $M$ is a solid torus. Then $\pi_1(M)=\Z$, and so $\pi_1(M)$ is not acylindrically arboreal. The manifold $M$ also contains no $2$-sided incompressible closed surface, so the result follows in this case. We will thus assume from now on that $M$ is not a solid torus.

    Now let $G=\pi_1(M)$. Then as $M$ is hyperbolic $G$ admits a natural action by isometries on $\mathbb{H}^3$ whose quotient is the interior of $M$. As such, we will refer to an element of $g\in G$ as \emph{hyperbolic} if it fixes exactly two points in the ideal boundary $\partial\mathbb{H}^3$ or \emph{peripheral} if it fixes exactly one point in $\partial \mathbb{H}^3$. By Theorem~\ref{thm:MrelHyp}, $G$ is hyperbolic relative to the subgroups that arise from the boundary components of $M$, so this definition of peripheral will agree with the definition of peripheral above.
    
    For the if direction assume that $M$ contains an embedded $2$-sided incompressible geometrically finite closed surface $\Sigma$ that is not isotopic to any boundary component of $M$. Then $\Sigma$ is not a disc as it is closed, $\Sigma$ is not a projective plane by the remark following Theorem~\ref{thm:HempelIBundles}, and $\Sigma$ is not a sphere by Theorem~\ref{thm:ired}, so $\pi_1(\Sigma)$ (and hence $G$) is infinite by the classification of surfaces (see \cite[Theorem~4.14]{kinsey1997topology}, for example). Furthermore, $\pi_1(\Sigma)$ is finitely generated as $\Sigma$ is closed. Since $\Sigma$ is geometrically finite we have that $H=\pi_1(\Sigma)$ is relatively quasi-convex by Theorem~\ref{thm:GFisQC}, and so by Theorem~\ref{thm:FinRelH} there exists a natural number $n_H$ such that the intersection of $n_H$ essentially distinct conjugates of $H$ in $G$ is either finite or a peripheral subgroup of $G$. In fact, it follows directly from the definition of relatively quasi-convex and by considering the action of $G$ on $\partial\mathbb{H}^3$ that $H$ is hyperbolic relative to the subgroups that arise as the intersections of $\Sigma$ with $\partial M$, and that all peripheral elements of $G$ that are included in $H$ are conjugate in $H$ to some boundary component of $\Sigma$. Since $\Sigma$ is closed and therefore does not intersect with $\partial M$, $H$ contains no peripheral elements and so the intersection of any collection of $n_H$ essentially distinct conjugates of $H$ in $G$ must be finite. However, $G$ is infinite and $M$ is orientable and irreducible, and so Lemma~\ref{lem:tf} tells us that $G$ and hence $H\leq G$ is torsion free, and thus the intersection of any collection of at least $n_H$ essentially distinct conjugates of $H$ in $G$ must be trivial.

    Let $(\Gamma, \mathfrak{G})$ be the splitting of $G$ induced by cutting $M$ along $\Sigma$, $T=T((\Gamma, \mathfrak{G})$ its Bass-Serre tree, and let  $p$ be a path in $T$ of length at least $n_H$. The stabiliser for this path will be the intersection of the stabilisers of its constituent edges, which are essentially distinct conjugates of $H$ in $G$, and so $\Ps_G(p)$ is trivial, and the action of $G$ on $T$ is therefore $(n_H, 1)$-acylindrical. The group $G$ is not virtually cyclic as it contains the fundamental group of $\Sigma$, a closed surface that is not a sphere or a projective plane. Thus the action of $G$ on $T$ is not lineal by Theorem~\ref{Thm:categorisedActions}, and to show that it is non-elementary it only remains to check that it is non-trivial.
 
    Assume for contradiction that this splitting is trivial, so $G_v=G$ for some vertex $v\in V(\Gamma)$. The graph $\Gamma$ has one edge, which corresponds to $\Sigma$, and so $\Gamma$ either has one vertex or two vertices. The former of these cannot occur, else $G$ would be an HNN extension of $G_v$, so $G_v\neq G$. We may thus assume that the splitting $(\Gamma, \mathfrak{G})$ is an amalgam decomposition $G\cong G_v*_{H}A$ for some subgroup $A$,  with $\pi_1(\Sigma)=H=A$. Since $\Sigma$ is $2$-sided it must locally separate $M$, and so since the cut along $\Sigma$ induces an amalgam we must have that $\Sigma$ separates $M$ into two components $M_1$ and $M_2$ with $A=\pi_1(M_1)$ and $G_v=\pi_1(M_2)$. By assumption we have that $\pi_1(\Sigma)=\pi_1(M_1)$, $\Sigma$ is a connected component of $\partial M_1$ that is not a sphere, projective plane or disc, and $M_1$ contains no $2$-sided projective plane as it is a submanifold of $M$ which itself contains no such subsurface. Furthermore, $M_1$ is irreducible as any incompressible sphere in $M_1$ would represent an incompressible sphere in $M$, and so we may apply Theorem~\ref{thm:HempelIBundles} to see that $M_1\cong \Sigma\times I$. This implies that $M_1$ has exactly one other boundary component $\Sigma'$ that is isotopic to $\Sigma$, and we therefore have that $\Sigma$ is isotopic to $\Sigma'$ in $M$. This gives the desired contradiction, as $\Sigma$ is not isotopic to any boundary component of $M$ by hypothesis. 
    
    It follows that this splitting is non-trivial, so the action of $\pi_1(M)$ on $(\Gamma, \mathfrak{G})$ is non-elementary by Theorem~\ref{Thm:categorisedActions} as required, and thus $\pi_1(M)$ is an acylindrically arboreal group.

    \bigskip For the only if direction assume that $\pi_1(M)$ is acylindrically arboreal and let $(\Gamma, \mathfrak{G})$ be a graph of groups splitting for $G=\pi_1(M)$ such that the action of $G$ on $T=T(\Gamma, \mathfrak{G})$ is non-elementary and acylindrical, which by Lemma~\ref{lem:AAimpliesFinite} we can assume has exactly one orbit of edges. The structure of our argument can be outlined in the following steps.
    \begin{enumerate}
        \item First we will apply Theorem~\ref{Thm:CullerShalen} to the splitting $(\Gamma, \mathfrak{G})$ to construct a set $\overline{\Sigma}=\{\Sigma_1,...,\Sigma_n\}$ of closed and $2$-sided incompressible surfaces in $M$.
        \item We will assume for contradiction that the fundamental groups of all of these surfaces are individually virtual surface fibre subgroups of $G$, and we will use this assumption to infer that at least one edge group of the splitting $(\Gamma, \mathfrak{G})$ must have been a virtual surface fibre subgroup of $G$.
        \item Finally we will apply Theorem~\ref{thm:mfdCateg} to see that such an edge group must be normal in $G$. This will lead to a contradiction to the fact that our original splitting $(\Gamma', \mathfrak{G}')$ was acylindrical.
    \end{enumerate}

    \textbf{Step 1 : }The action of $G$ on $T$ has exactly one orbit of edges by assumption, so we may apply Theorem~\ref{Thm:CullerShalen} to see that there exists a non-empty system \[\overline{\Sigma}=\{\Sigma_1,...,\Sigma_n\}\] of compact $2$-sided incompressible surfaces embedded in $M$ none of which are boundary parallel such that $\Img(\pi_1(\Sigma_i)\rightarrow \pi_1(M))$ for all $i$ is contained in some edge group of $(\Gamma, \mathfrak{G})$ and $\Img(\pi_1(M_j)\rightarrow \pi_1(M))$ for all connected components $M_j$ of $M\backslash\Sigma$ is contained in some vertex group of $(\Gamma, \mathfrak{G})$. Furthermore, since the boundary components of $M$ are tori, their subgroups are copies of $\Z^2$ and so must act elliptically on $T(\Gamma, \mathfrak{G})$ by Lemma~\ref{lem:BrokenProducts} using the acylindricity assumption. Thus by the second part of Theorem~\ref{Thm:CullerShalen} we can assume that all of our $\Sigma_i$'s are disjoint from the boundary of $M$, and therefore closed.

    Assume first that $M$ has boundary, so any finite sheeted cover of $M$ must also have boundary. Thus, if some cover $M'$ of $M$ fibres it must fibre over a surface with at least one boundary component. All virtual surface fibre subgroups of $G$ are therefore the fundamental groups of orientable surfaces with boundary, so must be free groups as such surfaces admit deformation retractions onto graphs. However, each surface in $\overline{\Sigma}$ may be assumed to be closed by the previous paragraph, and these surfaces are incompressible so cannot be spheres by Theorem~\ref{thm:ired}. The fundamental group of any surface in $\overline{\Sigma}$ will therefore be a finitely generated (from the fact that each surface is compact) group that is not a free group owing to the fact that such groups have cohomological dimension $2$ \cite[Chapter~VIII.2, Example~3]{brown_cohomology} but free groups have cohomological dimension $1$ \cite[Chapter~VIII.2, Example~2]{brown_cohomology} (see \cite[Chapter~VIII]{brown_cohomology} for a definition of cohomological dimension). Thus no surface in $\overline{\Sigma}$ has a fundamental group that includes into $G$ as a virtual surface fibre subgroup, and so they must all be geometrically finite by Theorem~\ref{thm:tameness}. The case where $M$ has non-empty boundary follows, and we will therefore assume for the remainder of this proof that $M$ is closed.

    \textbf{Step 2 : } Assume for contradiction that for all $i\in\{1,...,n\}$ the subgroup $\pi_1(\Sigma_i)$ is a virtual surface fibre subgroup of $G$, and fix some $i$. By construction $\pi_1(\Sigma_i)$ will be contained in some edge stabiliser of the action of $G$ on $T$, which will be some conjugate $G_e'$ of $G_e$ where $e$ is the unique edge of $\Gamma$. The subgroup $G_e'$ is an edge stabiliser in a non-elementary and thus non-trivial splitting, so cannot have finite index in $G$, and therefore we may invoke Lemma~\ref{lem:poisonVF} to see that $G_e'$ is a virtual surface fibre subgroup of $G$, and then $G_e$ is conjugate to $G_e'$ in $G$, so it must also be a virtual surface fibre subgroup.

    \textbf{Step 3 : }We now observe that the splitting of $G$ over $G_e$ is a non-trivial splitting of the fundamental group of a closed and orientable $3$-manifold with empty boundary over a virtual surface fibre subgroup, so we may invoke Theorem~\ref{thm:mfdCateg} to see that $G_e$ is a normal subgroup of $G$. Thus, the stabiliser of any lift of $e$ into $T$ must be equal to $G_e$, and since $e$ is the unique edge of $\Gamma$ every edge of $T$ is a lift of $e$ and $G_e$ fixes every edge of $T$ and must lie in the kernel of the action of $G$ on $T$. However, the tree $T$ is unbounded as it is the Bass-Serre tree of a non-trivial splitting, so this contradicts the acylindricity of the action.

    It follows that our original assumption that for $i\in\{1,..,n\}$, $\pi_1(\Sigma_i)$ is a virtual surface fibre subgroup of $G$ was false, and so by Theorem~\ref{thm:tameness} (subgroup tameness) and the fact that each $\pi_1(\Sigma_i)$ is the fundamental group of a closed surface and thus finitely generated, there must exist $i$ such that $\pi_1(\Sigma_i)$ is a geometrically finite subgroup. The surface $\Sigma_i$ was constructed using Theorem~\ref{Thm:CullerShalen}, and so is not boundary parallel, and $M$ must contain an embedded $2$-sided incompressible closed geometrically finite subsurface that is not boundary parallel as required.\qedhere
\end{proof}
\subsection{Proof of Theorem~\ref{thm:intromfd}}
In this section we use Theorems~\ref{thm:mainHypMfd} and~\ref{thm:mfdCateg} to prove Theorem~\ref{thm:intromfd}, and then briefly show that a similar strong classification result cannot extend to the case where we allow boundary. Theorem~\ref{thm:intromfd} is restated as follows.
\intromfd*
\begin{proof}
    Let $M$ be a closed and orientable hyperbolic $3$-manifold with fundamental group $G$. That (1) implies (2) is a standard result for hyperbolic groups, see \cite{Linton22} for example, but we include a brief proof for completeness. Assume that $G$ admits a non-elementary quasi-convex splitting over some quasi-convex subgroup $H$, with corresponding Bass--Serre tree $T$, and recall that by Theorem~\ref{thm:MrelHyp} $G$ is hyperbolic relative to the trivial subgroup. Then $H$ is by definition relatively quasi-convex with respect to the trivial subgroup, and so by Theorem~\ref{thm:FinRelH} there exists a natural number $n_H$ such that the intersection of at least $n_H$ distinct conjugates of $H$ must be a finite or peripheral subgroup. However, all of our peripheral subgroups are trivial in this case, so such an intersection must be finite, and by Lemma~\ref{lem:tf} and the fact that $G$ admits a non-elementary and thus non trivial quasi-convex splitting and is therefore torsion free we must have that $G$ is torsion free. Thus the intersection of at least $n_H$ essentially distinct conjugates of $H$ is trivial. Now let  $p$ be a path in $T$ of length at least $n_H$. The stabiliser for this path will be the intersection of the stabilisers of its constituent edges, which are essentially distinct conjugates of $H$ in $G$, and so $\Ps_G(p)$ is trivial, and the action of $G$ on $T$ is therefore $(n_H, 1)$-acylindrical. This action was non-elementary by assumption, and so it follows that $G$ is acylindrically arboreal, and (1) implies (2). 
    
    Now assume that (2) holds. By Theorem~\ref{thm:mainHypMfd}, $M$ contains an embedded $2$-sided incompressible closed subsurface $\Sigma$ that is not isotopic to any boundary component of $M$, and such that the image of the natural inclusion $\pi_1(\Sigma)\hookrightarrow\pi_1(M)$ is geometrically finite. Therefore, by Theorem~\ref{thm:GFisQC}, $\pi_1(\Sigma)$ must be relatively quasi-convex with respect to the boundary subgroups of $G$. The manifold $M$ has empty boundary, so $G$ is hyperbolic and $\pi_1(\Sigma)$ is a quasi-convex subgroup of $G$. It follows that the cut of $G$ along $\Sigma$ will induce a quasi-convex splitting of $G$. This splitting will be non-trivial by construction of $\Sigma$, and will be acylindrical by the proof that (1) implies (2). The group $G$ is not virtually cyclic as it contains a non-trivial surface subgroup, and so by Theorem~\ref{Thm:categorisedActions} the cut of $M$ along $\Sigma$ induces a non-trivial splitting, which will then be a non-trivial quasi-convex splitting as required. It follows that (2) implies (1), and so (1) and (2) are equivalent as claimed. 
    
    That (2) implies (3) follows by a simple contrapositive argument, so it only remains to show that (3) implies (2). We proceed once again by the contrapositive. As such, assume that $G$ is not acylindrically arboreal, so by Theorem~\ref{thm:mainHypMfd} $M$ contains no closed $2$-sided geometrically finite incompressible subsurface. If no non-trivial splitting of $G$ exists then $G$ has (FA$^-$), so assume that $G$ admits a non-trivial graph of groups decomposition $(\Gamma, \mathfrak{G})$. Let $T$ be the Bass--Serre tree of this splitting. Similarly to the proof of Theorem~\ref{thm:mainHypMfd}, we will show that the splitting $(\Gamma, \mathfrak{G})$ must have at least one edge whose associated group is a virtual surface fibre subgroup of $G$, and then use the categorisation in Theorem~\ref{thm:mfdCateg} to show that the existence of such an edge group guarantees a line $L$ in $T$ fixed setwise by the entire action of $G$ on $T$. 

    Let $e$ be any edge of $(\Gamma, \mathfrak{G})$ such that the splitting $(\Gamma', \mathfrak{G}')$ of $G$ over $G_e$ constructed in the proof of Lemma~\ref{lem:AAimpliesFinite} is non-trivial. Then $\Gamma'$ has exactly one edge (labelled $e$) and either one or two vertices, so by Theorem~\ref{Thm:CullerShalen} there exists at least one compact $2$-sided incompressible surface $\Sigma$ in $M$ that is not boundary parallel such that $\pi_1(\Sigma)\leq G$ is contained in some conjugate $G_e'$ of $G_e$.

    The subgroup $\pi_1(\Sigma)$ is the fundamental group of a compact surface so is finitely generated. Furthermore, by assumption and by Theorem~\ref{thm:mainHypMfd}, $\pi_1(\Sigma)$ is not geometrically finite so therefore must be a virtual surface fibre subgroup by Theorem~\ref{thm:tameness} (subgroup tameness). Thus, by Lemma~\ref{lem:poisonVF}, $G_e'$ is either a virtual surface fibre subgroup of $G$ or has finite index in $G$. The latter case cannot occur as $G_e$ is an edge group in a non-trivial splitting and so is not finite index and hence not conjugate to a finite index subgroup, and so $G_e'$ is in fact a virtual surface fibre subgroup of $G$. Therefore $G_e$ is conjugate to some virtual surface fibre subgroup of $G$ and is therefore is a virtual surface fibre subgroup itself.

    By Theorem~\ref{thm:mfdCateg} it follows that either $M$ fibres over $S^1$ with fibre equal to a surface with fundamental group $G_e$, or $M$ is a pair of twisted $I$-bundles $M_1$ and $M_2$ identified along their boundary, with $H$ equal to the subgroup of the fundamental group corresponding to the boundary component of either $I$-bundle, and in both cases $G_e$ is normal in $G$. Thus $G_e$ must act trivially on the convex hull $\overline{G\cdot e}$ of the orbit of $e$ in $T$, the Bass--Serre tree of our original splitting $(\Gamma, \mathfrak{G})$.
    
    The action of $G$ on $T$ must fix $\overline{G\cdot e}$ setwise, and the induced action of $G$ on the subtree $\overline{G\cdot e}$ factors through the quotient $G/G_e$. In the first case, where $G_e$ is a surface fibre subgroup of $G$, this quotient will be a copy of $\Z$ by definition, so must act on $\overline{G\cdot e}\subseteq T$ with a fixed point or fixed line, and thus the action of $G$ on $T$ will have a fixed point or fixed line. In the second case, where $M$ is two twisted $I$-bundles glued along their boundaries, we have that
    \[G/G_e\cong \left(\frac{\pi_1(M_1)}{G_e}\right)*_{G_e/G_e}\left(\frac{\pi_1(M_2)}{G_e}\right)\cong D_\infty,\]
    with the last congruence following from the fact that $G_e$ has index two in both $\pi_1(M_1)$ and $\pi_2(M)$ by Theorem~\ref{thm:HempelIBundles}. The group $D_\infty$ is virtually cyclic so must act on $\overline{G\cdot e}\subseteq T$ with a fixed point or fixed line, and thus the action of $G$ on $T$ will have a fixed point or fixed line.

    Therefore, since $(\Gamma, \mathfrak{G})$ was chosen arbitrarily it follows that $G$ must have property (FA$^-$) as required, and thus (3) implies (2) and all three conditions are equivalent as required.
\end{proof}

We finish this section by showing by means of an example that Theorem~\ref{thm:intromfd} cannot be extended to include the compact case.

\begin{example}
    Let $M=S^3-4_1$, the figure 8 knot complement in $S^3$, which fibres over the circle with fibre homeomorphic to a punctured torus. Then $M$ is hyperbolic and has the following properties.
    \begin{enumerate}
        \item By a result of Floyd and Hatcher \cite[Theorem 1.1]{FLOYD1982263} all closed incompressible embedded surfaces in $M$ are isotopic to the torus boundary component. The fundamental group $\pi_1(M)$ is therefore not acylindrically arboreal by Theorem~\ref{thm:mainHypMfd}.
        \item By the same theorem of Floyd and Hatcher, $M$ does contain several incompressible embedded surfaces, at least two of which are not isotopic to a fibre of the above fibration.
        \item Finally, the first singular homology group of $M$ is a copy of $\Z$, so any incompressible surface that induces a splitting of $M$ with a fixed line must be isotopic to the fibre of the above fibration.
    \end{enumerate}
    Properties (2) and (3) imply that $\pi_1(M)$ does not have (FA$^{-}$), as the cut along the surfaces in $M$ not isotopic to the fibre must give an interesting splitting. It follows that $M$ is a compact hyperbolic $3$-manifold with toroidal boundary such that $\pi_1(M)$ is not acylindrically arboreal, but $\pi_1(M)$ also does not have (FA$^-$).
\end{example}